\documentclass{article}
\usepackage{amsmath,amscd,amssymb,amsthm,tikz,tikz-cd,mathtools}
\usepackage[mathscr]{euscript}
\usepackage{tikz-cd, graphicx}
\usepackage{color}
\usepackage{cleveref}
\usepackage{wrapfig}
\usepackage{multirow}
\usepackage{url}
\usepackage{stmaryrd}

\tikzset{
  symbol/.style={
    draw=none,
    every to/.append style={
      edge node={node [sloped, allow upside down, auto=false]{$#1$}}}
  }
}

\newcommand{\CC}{\mathbb{C}}
\newcommand{\NN}{\mathbb{N}}
\newcommand{\QQ}{\mathbb{Q}}
\newcommand{\RR}{\mathbb{R}}
\newcommand{\ZZ}{\mathbb{Z}}

\newcommand{\Sym}{\textrm{Sym}}

\newcommand{\fin}{\textrm{fin}}
\newcommand{\Ab}{{\mathscr{A}\textrm{b}}}
\newcommand{\Mon}{{\mathscr{M}\textrm{on}}}
\newcommand{\Set}{{\mathscr{S}\textrm{et}}}
\newcommand{\Qdl}{{\mathscr{Q}\textrm{dl}}}
\newcommand{\Grp}{{\mathscr{G}\textrm{rp}}}

\newcommand{\Hom}{\textrm{Hom}}
\newcommand{\Aut}{\textrm{Aut}}
\newcommand{\Id}{\textrm{Id}}
\newcommand{\iso}{\xrightarrow{\sim}}

\newcommand{\Conj}{\textrm{Conj}}
\newcommand{\Inn}{{\textrm{Inn}}}
\newcommand{\ord}{{\textrm{ord}}}

\newcommand{\qq}[2]{{}^{#1}{#2}}
\newcommand{\bclose}[1]{\overline{#1}}
\newcommand{\Triv}{\textrm{Triv}}

\newcommand{\Mongr}{{\Mon_{gr}}}
\newcommand{\Setgr}{{\Set_{gr}}}
\newcommand{\Modgr}[1]{{\mathscr{M}\textrm{od}^{gr}_{#1}}}

\newcommand{\dom}{\textrm{dom}}
\newcommand{\Col}{\textrm{Col}}
\newcommand{\col}{\textrm{col}}
\newcommand{\Cov}{\textrm{Cov}}

\newcommand{\genfunc}{\eta}
\newcommand{\A}{\mathscr{A}}
\newcommand{\m}{\mathfrak{m}}

\DeclareMathOperator{\im}{Im}

\newcommand{\Dhd}[1]{{D_{#1}}}
\newcommand{\T}{\textrm{T}}

\newcommand{\Zlx}{{\left(\ZZ/\ell\ZZ\right)^\times}}

\newtheorem{theorem}{Theorem}[section]
\newtheorem{proposition}[theorem]{Proposition}
\newtheorem{corollary}[theorem]{Corollary}
\newtheorem{lemma}[theorem]{Lemma}

\theoremstyle{definition}
\newtheorem{definition}[theorem]{Definition}
\newtheorem{example}[theorem]{Example}
\newtheorem{remark}[theorem]{Remark}

\title{The Hilbert Polynomial of Quandles and Colorings of Random Links}
\author{Ariel Davis, Tomer M. Schlank}
\begin{document}
\maketitle
\abstract{Given a finite quandle $Q$,  we study the average number of $Q$-colorings of the closure of a random braid in $B_n$ as  $n$ varies.  In particular we show that this number coincides with some polynomial $P_Q\in \QQ[x]$ for  $n\gg 0$. The degree of this polynomial is readily computed in terms of $Q$ as a quandle and these invariants are computed for all quandles with $|Q|\le 4$. Additionally we show that the methods in this paper allow to improve on the stability results of \cite{WEBSITE:ElVenWest2015} from ``periodic stability'' to ``stability''.}

\section{Introduction}
A quandle is an algebraic structure particularly well-suited for producing numerical invariants of oriented links: For a finite quandle $Q$ and oriented link $L$ in $S^3$, one counts $\Col_Q(L)$ - the set of ``$Q$-colorings" of $L$. 
Our main goal in this paper is to study the average size $\Col_Q(L)$ where $L$ is chosen randomly due to a certain family of distributions. Informally we will say that we are interested in the average size of  $\Col_Q(L)$ where $L$ is taken to be the closure $\bclose\sigma$ of a random braid $\sigma \in B_n$ of $n$-strands.  The formal meaning is not immediately clear as $B_n$ is a countable discreet group. However it is not hard to show (See \Cref{Prp: c hat}) that the function $c_Q\colon B_n \to \NN$ sending $\sigma \in B_n$ to  $c_Q(\sigma) \coloneqq |\Col_Q(\bclose\sigma)|\in \NN$ extends uniquely to a continues function  $c_Q\colon \widehat{B_n} \to \NN$ from the profinite completion of $B_n$. Since $\widehat{B_n}$ is compact and Hausdorff we get a well defined average $\mathbb{E}_{\widehat{B_n}}c_Q = \int\limits_{\widehat{B_n}}c_Q d\mu$. The main result of this paper is:
\begin{theorem}[\Cref{Thm: Qdl HPoly}]\label{thm:Intro}
	Let $Q\in\Qdl^\fin$. Then there exist integer-valued polynomials $P_Q(x)\in\QQ[x]$ s.t. 
	$$\int\limits_{\widehat{B_n}}c_Q d\mu=P_Q(n)$$
	for all $n\gg 0$.
\end{theorem}
We then give a give a simple formula for $\deg P_Q$  in term of the number of connected components of sub-quandles of $Q$ (\cref{Prp: dim dom}).

\Cref{thm:Intro} is reminiscent  of the existence of Hilbert polynomial for graded modules in commutative algebra. Indeed, the main technical  step   in the proof is \Cref{Prp: good HPoly} that gives a  generalization of  Hilbert polynomial to graded modules over certain class of well behaved non-commutative graded rings.    \Cref{Prp: good HPoly} can be use also to give a slight improvement on the stability result in \cite{WEBSITE:ElVenWest2015}. Specifically \cite[Thm 6.1]{WEBSITE:ElVenWest2015}  gives that dimension of the  homology groups $\dim_\CC H_i(\textrm{Hur}_{G,n}^c,\CC)$ of  certain sequence of Hurwitz spaces is eventually periodic  in $n$,  \Cref{Prp: good HPoly} gives that this sequence is in fact eventually \textbf{constant} (see \Cref{remEVW}).

\subsection{The Contents of this Paper}
In section $2$ we begin with preliminary facts about quandles. We show that $Q$-coloring links produces a random variable $c_{Q,n}$ on $\widehat{B_n}$, the profinite braid group on $n$ strands. The expected value is described explicitly in terms of $Q$ and $n$.

In section $3$ the expected values of $c_{Q,n}$, ranging over all $n$, are recovered as the graded cardinalities of $\A_Q$, a graded monoid associated to $Q$. We then prove the main result that there is an integer-valued polynomial $P_Q\in\QQ[x]$ s.t. the expected valued of $c_{Q,n}$ equals $P_Q(n)$ for all large enough $n$. We call this the \emph{Hilbert polynomial} of $Q$. 

In section $4$ we compute the degree of $P_Q$: $\deg P_Q=\dim_Q-1$, where $\dim_Q$ is an invariant directly obtained from the structure of $Q$.

In section $5$ we show that $\dim_{Q\times R}=\dim_Q\cdot \dim_R$. While the previous sections focused on the expected values of $Q$-colorings, the observation that $c_{Q\times R}=c_Q\cdot c_R$ allows one to recover more refined statistical information such as higher moments and covariants in terms of the very same machinery - applied to larger quandles.

In section $6$ we discuss disjoint unions of quandles. The generating function $\eta_Q(t)=\sum_n |\A_{Q,n}|t^n\in\ZZ[[t]]$ of $Q$ is a more refined invariant than $P_Q(x)$. We show that $\eta_{Q\sqcup R}(t)=\eta_Q(t)\cdot \eta_R(t)$. This offers a simple computation of $P_{Q\sqcup R}(x)$, which is generally not determined by $P_Q(x)$ and $P_R(x)$.

In section $7$ we compute examples of Hilbert polynomials of several families of quandles, including all finite quandles up to size $4$.
\subsection{Acknowledgements}
The authors would like to express their gratitude to Ruth Lawrence, Ohad Feldheim and the entire Seminarak group for useful discussions. We would also like to thank Aaron Landesman for suggesting the relevance of our methods to \cite{WEBSITE:ElVenWest2015}. The second author is supported by ISF1588/18 and BSF 2018389.
\section{From Links and Quandles to Monoids}

\subsection{Established facts about quandles}
This section contains statements about quandles readily found in the literature. The definition and algebraic properties of quandles are out of {\cite{ARTICLE:Joyce1982}}, as is the definition of the fundamental quandle of an oriented link. The $B_n$-action on $n$-tuples is taken from \cite{ARTICLE:Brieskorn1986}.

\subsubsection{The Fundamentals of Quandles - Algebraic Notions}

Recall in {\cite{ARTICLE:Joyce1982}} the definition of a quandle. We find the notation of \cite[\textsection 1]{ARTICLE:FennRourke1992} preferable.
We also choose to have quandles act on the left as in \cite{ARTICLE:Brieskorn1986}. 
\begin{definition}
	\label{Dfn: quandle}
	A \emph{quandle} is a set $Q$ with a binary operator
	$$Q\times Q\to Q\;,\;\;\;\;(x,y)\mapsto \qq xy$$
	satisying the following:
	\begin{enumerate}
		\item[(Q1)] for every $x\in Q$, $\qq xx=x$.
		\item[(Q2)] for every $x\in Q$, $\qq x{(-)}\colon Q\to Q$ is a bijection.
		\item[(Q3)] for every $x,y,z\in Q$, $\qq x{(\qq yz)}=\qq{(\qq xy)}{(\qq xz)}$.
	\end{enumerate}
	A \emph{morphism} of quandles is a function $f\colon Q\to R$ s.t. for all $x,y\in Q$,
	$$f(\qq xy)=\qq{f(x)}{f(y)}.$$
\end{definition}
All quandles and quandle morphisms naturally
form a category, denoted here by $\Qdl$. A quandle morphism is an isomorphism iff it is a bijection.
Let $Q\in\Qdl$ and let $x\in Q$. 
By $\varphi_x\colon Q\to Q$ we denote the function 
$\varphi_x(y)=\qq xy$. By axioms (Q2) and (Q3) of \cref{Dfn: quandle}, for all $x\in Q$
$$\varphi_x\in\Aut_\Qdl(Q).$$

We denote by $\Qdl^\fin\subseteq\Qdl$ the full subcategory of finite quandles, that is, those $Q\in\Qdl$ with finite underlying set.

\begin{example}[{\cite[\textsection1]{ARTICLE:Joyce1982}}]
	Let $G\in\Grp$ be a group. 
	Then there is a canonical quandle structure on $G$ 
	given by conjugation: 
	$\qq hg:=hgh^{-1}$ satisfies all quandle axioms. 
	There is a functor
	$$\Conj\colon\Grp\to\Qdl\;\;,\;\;\;\Conj(G)=(G,\qq gh=ghg^{-1})\;,\;\;\Conj(f\colon G\to H)=f.$$
	In fact, let $G\in\Grp$ and let $S\subseteq G$ be a subset closed under conjugation. Then $S$ is a quandle with structure induced from conjugation in $G$. For example, any conjugacy class or union of such, or any singleton. We will often write $G\in\Qdl$ in place of $\Conj(G)$ when confusion is unlikely.
\end{example}

The following notion of a trivial quandle can be found in many sources, for example \cite{BOOK:Nosaka}.
\begin{example}
	\label{Ex: Triv Qdl}
	For every set $X\in\Set$, we denote by $\Triv(X)\in\Qdl$ the \emph{trivial quandle} with underlying set $X$ and
	$\varphi_x=\Id_X$
	for all $x\in X$:
	$$\forall\;x,y\in X\;,\;\;\qq xy=y.$$
	Let $X,Y\in\Set$. Every function $f\colon X\xrightarrow{f} Y$ in $\Set$ is a quandle morphism between trivial quandles:
	$$f(\qq xy)=f(y)=\qq{f(x)}{f(y)},$$
	making $\Triv\colon\Set\to\Qdl$ into a functor. For $\alpha\in\NN$ we denote by $T_\alpha\in\Qdl$ the \emph{trivial quandle with $\alpha$ elements}, which is well-defined up to isomorphism.
\end{example}

\begin{definition}
	Let $Q\in\Qdl$. A \emph{sub-quandle} $R$ of $Q$ is a subset $R\subseteq Q$ s.t.
	$$\forall \;x,y\in R\;,\;\;\qq xy\in R.$$
	For such $R$ we use the notation $R\le Q$. 
	For $R\le Q$ and $R\neq Q$, we write $R<Q$.
\end{definition}

\begin{proposition}
	\label{Prp: im sub}
	Let $f\colon R\to Q$ be a quandle morphism. Then
	$$f(R)\le Q.$$
\end{proposition}
\begin{proof}
Let $f(x),f(y)\in f(R)$ for some $x,y\in R$. Then
$$\qq{f(x)}{f(y)}=f(\qq xy)\in f(R).$$
\end{proof}

\begin{definition}
	Let $Q\in\Qdl$ and let $S\subseteq Q$ be a subset of $Q$. 
	Then $\langle S\rangle\le Q$ is the \emph{sub-quandle generated by $S$}, the minimal $R\le Q$ containing $S$.
	In the case $Q=\Conj(G)$ for some $G\in\Grp$,
	we distinguish between the sub-quandle and subgroup generated by $S$ with respective notation $\langle S\rangle_\Qdl$ and $\langle S\rangle_\Grp$.
\end{definition}

\begin{definition}[{\cite[\textsection 5]{ARTICLE:Joyce1982}}]
	\label{Dfn: Inn}
	Let $Q\in\Qdl$. Recall that 
	$\varphi_x\in\Aut_\Qdl(Q)$ for all $x\in Q$. The \emph{inner automorphism group} of $Q$ is defined as the subgroup of $\Aut_\Qdl(Q)$ generated by these $\varphi_x$:
	$$\Inn(Q):=\langle\{\varphi_x\}_{x\in Q}\rangle_\Grp\le\Aut_\Qdl(Q)$$
\end{definition}
Note that $\Inn(Q)$ from quandles to groups is \emph{not} functorial.

\begin{proposition}
	\label{Prp: Q to Inn(Q) is quandle morphism}
	Let $Q\in\Qdl$. Then
	$$Q\to\Conj(\Inn(Q))\;\;,\;\;\;\;x\mapsto \varphi_x$$
	is a morphism in $\Qdl$.
\end{proposition}
This proposition appears in {\cite[\textsection 2]{ARTICLE:Brieskorn1986}}. We include a proof for completeness.
\begin{proof}
	Let $x,y\in Q$. Then for all $z\in Q$,
	$$(\varphi_x\circ\varphi_y)(z)=\qq{x}{(\qq yz)}=\qq{(\qq xy)}{(\qq xz)}=(\varphi_{\qq xy}\circ\varphi_x)(z).$$
	Therefore
	$$\varphi_{\qq xy}=\varphi_x\circ\varphi_y\circ\varphi_x^{-1}=\qq{\varphi_x}{\varphi_y}\in\Inn(Q).$$
\end{proof}

In \cite{ARTICLE:Joyce1982}, a quandle $Q$ is called \emph{connected} if $\Inn(Q)$ acts transitively on $Q$. With this in mind we borrow the following notation from topology:

\begin{definition}
	\label{Dfn: pi0}
	Let $Q\in\Qdl$. We define the following set quotient
 $$\pi_0(Q):=Q/\Inn(Q).$$
 The elements of $\pi_0(Q)$ are the \emph{connected components} of $Q$.
\end{definition}

Since $\Inn(Q)$ is generated by $\{\varphi_x\}_{x\in Q}$, then $\pi_0(Q)$ is the quotient of the set $Q$ by the equivalence relation generated by $\qq xy\sim y$ for all $x,y\in Q$. This observation will be useful in proving the following lemma:

\begin{lemma}
	The map $Q\mapsto\pi_0(Q)$ is a functor
	$$\pi_0\colon \Qdl\to\Set.$$
\end{lemma}
\begin{proof}
	Let $f\colon R\to Q$ be a morphism of quandles. Then $f$ carries relations on $R$ of the form $\qq xy\sim y$ ($x,y\in R$) to relations of the same form on $Q$:
	$$\qq{f(x)}{f(y)}= f(\qq xy)\sim f(y).$$
	Therefore $f$ respects the equivalence relation generated by these, hence
	$$f\colon \pi_0(R)\to\pi_0(Q)$$
is well-defined. Functoriality is straightforward.
\end{proof}

\subsubsection{Oriented Links and Quandles}
The foundations for the modern study of knots and links were laid in the works of Alexander and Briggs, and Reidemeister - see \cite{ARTICLE:AlexBriggs1926} and \cite{BOOK:Reidemeister}. What follows here is a brief overview of the essential notions.

A \emph{link} $L$ in $S^3$ is an embedding $L\colon\coprod S^1\hookrightarrow S^3$ of a finite disjoint union of circles into the $3$-sphere. By orienting $S^3$ and $S^1$, one obtains the notion of an \emph{oriented} link. Two oriented links $L_1,L_2$ are \emph{equivalent} if there is an orientation-preserving homeomorphism $f\colon S^3\iso S^3$ s.t. $f\circ L_1=L_2$. An oriented link is \emph{tame} if it is equivalent to a polygonal link. For the entirety of this paper all links are assumed to be oriented and tame.

By suitably projecting a link $L$ to $\RR^2$, one obtains a \emph{link diagram} - a combinatorial presentation of the link comprised of oriented arcs and crossings. Each crossing, representing a $2$-to-$1$ fiber in $L$ of the projection, consists of one strand crossing over and one strand crossing under, split into two arcs. If the link is a knot, one typically sees the arcs numbered cyclically	- $x_1,\dots,x_n$.
For general link diagram $D$, the data of such a diagram might be organized in terms of a finite set $arcs(D)$ of oriented arcs in $\RR^2$, with additional data describing the crossings:
\begin{itemize}
	\item A permutation $n\colon arcs(D)\to arcs(D)$ assigning the \emph{next} arc.
	\item A function $c\colon arcs(D)\to arcs(D)$ - the arc crossing above $x$ and $n_x$.
	\item A function $\varepsilon\colon arcs(D)\to\{\pm1\}$ - the orientation at the crossing of $c_x$ over $x$ and $n_x$.
\end{itemize}

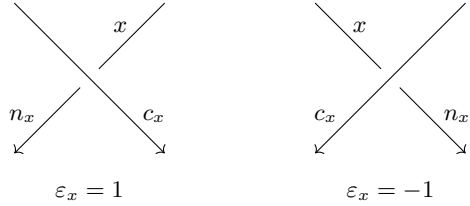
$$\begin{tikzpicture}
			\label{eqn: crossings}
			\draw[->] (2,2)--(0,0);
			\draw[-,line width=10pt, draw=white] (0,2)--(2,0);
			\draw[->] (0,2)--(2,0);
			\draw (1.5,0.5)node [anchor=west]{$\;c_x$};
			\draw (1.5,1.5)node[anchor = south]{$x\;\;$};
			\draw (0.5,0.5)node[anchor=east]{$n_x\;$};
			\draw (1,-0.5)node{$\varepsilon_x=1$};
			\draw[->] (4,2)--(6,0);
			\draw[-,line width=10=pt, draw=white] (6,2)--(4,0);
			\draw[->] (6,2)--(4,0);
			
			\draw (4.5,0.5)node[anchor = east]{$c_x\;$};
			\draw (5.5,0.5)node[anchor = west]{$\;n_x$};
			\draw (4.5,1.5)node[anchor = south]{$\;\;x$};
			\draw (5,-0.5)node{$\varepsilon_x=-1$};
	\end{tikzpicture}$$

As shown in \cite{BOOK:Reidemeister}, two link diagrams depict equivalent links iff they differ by a sequence of local moves on diagrams, called \emph{Reidemeister moves}.\\

The fundamental quandle of an oriented link $L$ is described in \cite[\textsection15]{ARTICLE:Joyce1982} in terms of generators and relations that are determined by a diagram $D$ depicting $L$:

\begin{proposition}
	\label{Prp: Q_D isotopy invar}
	Let $L$ be an oriented link in $S^3$. For a link diagram $D$ of $L$, let $Q_D\in\Qdl$ be the following quandle defined by generators and relations:
	$$Q_D=\left\langle\;x\in arcs(D)\;|\;\forall x\in arcs(D)\;,\;\;\varphi_{c_x}^{\varepsilon_x}(x)=n_x\;\right\rangle=$$
	$$= \left\langle\;x\in arcs(D)\;|\;\forall x\in arcs(D)
	\;,\;\;\begin{cases}
		\qq{c_x}x=n_x &,\;\varepsilon_x=1\\
		\qq{c_x}{n_x}=x &,\;\varepsilon_x=-1
	\end{cases}\;\right\rangle.$$
	Then up to isomorphism, $Q_D$ is a well-defined invariant of $L$ up to equivalence.
\end{proposition}

This is proved in \cite[\textsection15]{ARTICLE:Joyce1982} in the case where $L$ is a knot, by showing that if $D'$ and $D$ are link diagrams that differ by a Reidemeister move, then $Q_D\simeq Q_{D'}$. The proof extends verbatim for general oriented link $L$.

\begin{definition}
	Let $L$ be an oriented link in $S^3$. The \emph{fundamental quandle} $Q_L$ of $L$ is defined to be $Q_D\in\Qdl$  for any link diagram $D$ of $L$. By \cref{Prp: Q_D isotopy invar}, $Q_L$ is well-defined up to isomorphism.
\end{definition}

\begin{definition}
	Let $Q\in\Qdl$ and let $L$ be an oriented link in $S^3$. A \emph{$Q$-coloring of $L$} is a morphism $f\colon Q_L\to Q$ in $\Qdl$. A \emph{dominant} $Q$-coloring of $L$ is a surjective morphism. We denote by
	$$\Col_Q(L):=\Hom_\Qdl(Q_L,Q),$$
	$$\Col_Q^\dom(L)=\{f\in \Hom_\Qdl(Q_L,Q)\;|\;\im(f)=Q\}$$
\end{definition}

This presentation of $Q_L$ in terms of generators and relations enables the interpretation of a prior notion of \emph{knot coloring} - the ``coloring" of arcs of a diagram $D$, subject to certain conditions at each crossing - in terms of morphisms of quandles $Q_L\simeq Q_D\to Q$. More precisely:

\begin{proposition}
	Let $L$ be an oriented link with link diagram $D$. Then
	$$\Col_Q(L)\simeq\{f\colon arcs(D)\to Q\;|\;\forall x\in arcs(D)\;,\;\;\varphi_{f(c_x)}^{\varepsilon_x}(f(x))=f(n_x)\}:$$
	$$\begin{tikzpicture}
			\label{eqn: crossings}
			\draw[->] (2,2)--(0,0);
			\draw[-,line width=10pt, draw=white] (0,2)--(2,0);
			\draw[->] (0,2)--(2,0);
			\draw (1.5,0.5)node [anchor=west]{$\;f(c_x)$};
			\draw (1.5,1.5)node[anchor = south]{$f(x)\;\;\;\;$};
			\draw (0.5,0.5)node[anchor=east]{$f(n_x)\;$};
			\draw (1,-0.5)node{$\varepsilon_x=1$};
			\draw[->] (4,2)--(6,0);
			\draw[-,line width=10=pt, draw=white] (6,2)--(4,0);
			\draw[->] (6,2)--(4,0);
			
			\draw (4.5,0.5)node[anchor = east]{$f(c_x)\;$};
			\draw (5.5,0.5)node[anchor = west]{$\;f(n_x)$};
			\draw (4.5,1.5)node[anchor = south]{$\;\;\;\;\;f(x)$};
			\draw (5,-0.5)node{$\varepsilon_x=-1$};
	\end{tikzpicture}$$
\end{proposition}

\begin{proposition}
	Let $Q\in\Qdl$ and let $L$ be an oriented link in $S^3$. Then
	$$\Col_Q^\dom(L)=\Col_Q(L)\setminus\bigcup_{R<Q}\Col_R(L).$$
\end{proposition}
\begin{proof}
	By \cref{Prp: im sub}, for all $f\in\Col_Q(L)$, $\im(f)\le Q$. Therefore a $Q$-coloring $f$ is \emph{not} surjective iff its image is a proper sub-quandle of $Q$.
	Thus
	$$\Col_Q^\dom(L)=\Col_Q(L)\setminus\bigcup_{R<Q}\Col_R(L).$$
\end{proof}

\begin{proposition}
	\label{Prp: fin Q fin Col}
	Let $L$ be an oriented link in $S^3$ and let $Q\in\Qdl^\fin$ be a finite quandle. Then $\Col_Q(L),\Col_Q^\dom(L)$ are finite sets.
\end{proposition}
\begin{proof}
	Let $D$ be a link diagram for $L$, consisting of finitely many arcs. Then $\Col_Q(L)$ injects into $\Hom_\Set(arcs(D),Q)$, which is a finite set because $Q$ is finite. Thus $\Col_Q(L)$ is finite and so is $\Col_Q^\dom(L)\subseteq\Col_Q(L)$.
\end{proof}

\begin{definition}
Let $L$ be an oriented link in $S^3$ and let $Q\in\Qdl^\fin$. By \cref{Prp: fin Q fin Col}, $\Col_Q(L),\Col_Q^\dom(L)$ are finite. By $\col_Q(L),\col_Q^\dom(L)$ we denote their rective cardinalities:
	$$\col_Q(L):=\left|\Col_Q(L)\right|\in\NN,$$
	$$\col_Q^\dom(L):=\left|\Col_Q^\dom(L)\right|\in\NN.$$
\end{definition}

\subsubsection{Braids and the Artin Braid Group $B_n$}

Much like knots and links, braids admit both a topological definition and a combinatorial description. Recall in \cite[\textsection1.2.2]{BOOK:KassTur2008} the definition of a \emph{braid diagram} on $n$ strands - $n$ copies of the interval $I=[0,1]$ embedded into $\RR\times I$ s.t. each strand is a section of the projection $\RR\times I\to I$, the set of endpoints is fixed - $\{1,\dots n\}\times\{0,1\}$, no more than two strands may intersect at any point, where they intersect transversely. A braid diagram includes the data of which strand crosses over and under at any intersection. A braid diagram is depicted in terms of arcs - similarly to link diagrams, except for the fixed endpoints above and below. We orient braids downwards.

Concatenation of braids defines a group structure on $B_n$, the set of braids on $n$ strands. This group admits the following presentation in terms of generators and relations - see \cite[\textsection 1.1]{BOOK:KassTur2008}:
\begin{definition}
The Artin braid group $B_n$ on $n$ strands is given by generators and relations
$$B_n=\langle\sigma_1\dots\sigma_{n-1}\;|\; \sigma_i\sigma_{i+1}\sigma_i=\sigma_{i+1}\sigma_i\sigma_{i+1},\;\;[\sigma_i,\sigma_j]=1\textrm{ for }|i-j|\ge 2\rangle.
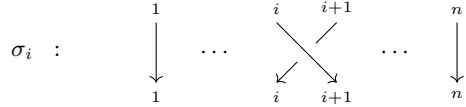
$$
	$$\begin{tikzpicture}[scale=0.8]
			\label{eqn: Brieskorn step}
			\draw (-4,0.5)node{$\sigma_i\;\;:$};
			\draw[->] (-2,1)node[anchor=south]{$\scriptstyle{1} $}--(-2,0)node[anchor=north]{$\scriptstyle{1}$};
			\draw (-1,0.5)node{$\cdots$};
			\draw[->] (1,1)node[anchor=south]{$\scriptstyle{i+1}$}--(0,0)node[anchor=north]{$\scriptstyle{i}$};
			\draw[-,line width=10pt, draw=white] (0,1)--(1,0);
			\draw[->] (0,1)node[anchor=south]{$\scriptstyle{i}$}--(1,0)node[anchor=north]{$\scriptstyle{i+1} $};
			
			\draw (2,0.5)node{$\cdots$};
			\draw[->] (3,1)node[anchor=south]{$\scriptstyle{n} $}--(3,0)node[anchor=north]{$\scriptstyle{n} $};
	\end{tikzpicture}$$
\end{definition}

One obtains an oriented link from braids via \emph{braid closure}. In terms of braid diagrams this is done by closing up the braid strands from bottom back up to the top such that they don't cross on the way. A precise definition may be found in  \cite[\textsection 2.2.3]{BOOK:KassTur2008}. Below depicted is an example:

$$\begin{tikzpicture}

\draw[->,rounded corners] (0.5,4)--(0.5,3.5)--(0,3)--(0,1.5)--(1,0.5)--(1,0);
\draw [-,line width=6 pt, draw=white] (0,4)--(0,3.5)--(1,2.5)--(0.5,2)--(0.5,1.5)--(0,1)--(0,0);
\draw[->, rounded corners] (0,4)--(0,3.5)--(1,2.5)--(0.5,2)--(0.5,1.5)--(0,1)--(0,0);
\draw [-,line width=6 pt, draw=white] (1,4)--(1,3)--(0.5,2.5)--(1.5,1.5)--(1.5,0);
\draw[->, rounded corners] (1,4)--(1,3)--(0.5,2.5)--(1.5,1.5)--(1.5,0);
\draw [-,line width=6 pt, draw=white] (1.5,4)--(1.5,2)--(1,1.5)--(1,1)--(0.5,0.5)--(0.5,0);
\draw[->, rounded corners] (1.5,4)--(1.5,2)--(1,1.5)--(1,1)--(0.5,0.5)--(0.5,0);

\draw (-0.5,2)node{$\sigma$};

\begin{scope}[xshift=5cm]
\draw[->,rounded corners] (0.5,4)--(0.5,3.5)--(0,3)--(0,1.5)--(1,0.5)--(1,0);
\draw [-,line width=6 pt, draw=white] (0,4)--(0,3.5)--(1,2.5)--(0.5,2)--(0.5,1.5)--(0,1)--(0,0);
\draw[->, rounded corners] (0,4)--(0,3.5)--(1,2.5)--(0.5,2)--(0.5,1.5)--(0,1)--(0,0);
\draw [-,line width=6 pt, draw=white] (1,4)--(1,3)--(0.5,2.5)--(1.5,1.5)--(1.5,0);
\draw[->, rounded corners] (1,4)--(1,3)--(0.5,2.5)--(1.5,1.5)--(1.5,0);
\draw [-,line width=6 pt, draw=white] (1.5,4)--(1.5,2)--(1,1.5)--(1,1)--(0.5,0.5)--(0.5,0);
\draw[->, rounded corners] (1.5,4)--(1.5,2)--(1,1.5)--(1,1)--(0.5,0.5)--(0.5,0);
\draw (0,0) arc (0:-180:0.5);
\draw (0.5,0) arc (0:-180:0.8);
\draw (1,0) arc (0:-180:1.1);
\draw (1.5,0) arc (0:-180:1.4);
\draw (-1,0)--(-1,1.5);
\draw[-] (-1,1.5)--(-1,2.5);
\draw (-1,2.5)--(-1,4);
\draw (-1.1,0)--(-1.1,1.5);
\draw[-] (-1.1,1.5)--(-1.1,2.5);
\draw (-1.1,2.5)--(-1.1,4);
\draw (-1.2,0)--(-1.2,1.5);
\draw[-] (-1.2,1.5)--(-1.2,2.5);
\draw (-1.2,2.5)--(-1.2,4);
\draw (-1.3,0)--(-1.3,1.5);
\draw[-] (-1.3,1.5)--(-1.3,2.5);
\draw (-1.3,2.5)--(-1.3,4);
\draw[<-] (0,4) arc (0:180:0.5);
\draw[<-] (0.5,4) arc (0:180:0.8);
\draw[<-] (1,4) arc (0:180:1.1);
\draw[<-] (1.5,4) arc (0:180:1.4);

\draw (-2,2)node{$\bclose\sigma$};
\end{scope}
\end{tikzpicture}$$

\begin{proposition}
	\label{Dfn: Brieskorn rep}
	Let $Q\in\Qdl$ and $n\in\NN$. Then the action of $\sigma_i\in B_n$ on the set $Q^n$ via
	$$\sigma_i(x_1\;\dots\;x_i\;,\;x_{i+1}\;\dots\;x_n)= (x_1\;\dots\;\qq{x_i}{x_{i+1}}\;,\;x_i\;\dots\;x_n)$$
        extends to a well-defined $B_n$-action on $Q^n\in\Set$.
\end{proposition}
This statement appears as part of proposition $3.1$ in {\cite[\textsection 3]{ARTICLE:Brieskorn1986}}. We include a proof for completeness.
\begin{proof}
	inverse to the $\sigma_i$-action is
	$$(x_1\;\dots\;x_i\;,\;x_{i+1}\;\dots\;x_n)\mapsto (x_1\;\dots\; x_{i+1}\;,\; \varphi_{x_{i+1}}^{-1}({x_i}) \;\dots\;x_n).$$
	Let $(x_1,\dots,x_n)\in Q^n$. For all $1\le i<j\le n-1$ s.t. $|i-j|\ge 2$,
	$$(\sigma_i\circ\sigma_j)(x_1,\dots,x_n)=
	(\dots\;\qq{x_i}{x_{i+1}}\;,\;x_i\; \dots\;\qq{x_j}{x_{j+1}}\;,\;x_j
	\;\dots)=$$$$=(\sigma_j\circ\sigma_i)(x_1,\dots,x_n).$$
	 Now let $1\le i\le n-2$. Taking $x:=x_i$, $y:=x_{i+1}$ and $z:=x_{i+2}$, then
	$$(\dots x,y,z\dots)\xmapsto{\sigma_{i}}
	(\dots\qq xy,x,z\dots)\xmapsto{\sigma_{i+1}}
	(\dots\qq xy,\qq xz,x\dots)\xmapsto{\sigma_{i}}
	(\dots\qq{(\qq xy)}{(\qq xz)},\qq xy,x\dots)$$
	$$(\dots x,y,z\dots)\xmapsto{\sigma_{i+1}}
	(\dots x,\qq yz,y\dots)\xmapsto{\sigma_{i}}
	(\dots\qq{x}{(\qq yz)},x,y\dots)\xmapsto{\sigma_{i+1}}
	(\dots \qq{x}{(\qq yz)},\qq xy,x\dots)$$
	These two coincide due to quandle axiom $(Q3)$:
	$\qq{x}{(\qq yz)}=\qq{(\qq xy)}{(\qq xz)}.$
	Therefore
	$$(\sigma_i\circ\sigma_{i+1}\circ\sigma_i)(x_1,\dots,x_n)= (\sigma_{i+1}\circ\sigma_i\circ\sigma_{i+1}) (x_1,\dots,x_n).$$
	The actions of the $\sigma_i$ respect the relations in $B_n$, thus extend to a $B_n$-action.
\end{proof}

\begin{definition}
	Let $n\in\NN$ and let $Q\in\Qdl$. We denote by 
 $$\beta_{Q,n}\colon B_n\to\Aut_\Set(Q^n)$$
 The homomorphism described in \cref{Dfn: Brieskorn rep}.	
\end{definition}

\begin{lemma}
\label{Lma: Q^n is B_n equivar}
	Let $f\colon R\to Q$ be a morphism in $\Qdl$ and let $n\in\NN$. Then the diagonal map $f_n\colon R^n\to Q^n$ commutes with the $B_n$ action.
\end{lemma}
\begin{proof}
We prove that for all $\sigma\in B_n$,
	$$\beta_Q(\sigma)\circ f_n= f_n\circ\beta_R(\sigma).$$
	It suffices to show this for generators: For all $1\le i\le n-1$,
	$$\beta_Q(\sigma_i)\circ f_n=f_n\circ\beta_R(\sigma_i)\colon R^n\to Q^n.$$
	Indeed for any $x,y\in R$ and 
	the appropriate $\sigma_i\in B_n$, we have 
	$$(\beta_Q(\sigma_i)\circ f_n)(\dots,x,y,\dots)=
	(\dots,\qq{f(x)}{f(y)},f(x),\dots)= $$
	$$(\dots,f(\qq xy),f(x),\dots)=
	(f_n\circ\beta_R(\sigma_i)) (\dots,x,y,\dots).$$
\end{proof}

\begin{definition}
	Let $Q\in\Qdl^\fin$ and let $n\in\NN$. We define the $B_n$-set $$Q^{n,\dom}:=Q^n\setminus\bigcup_{R<Q}R^n\in\Set_{B_n}.$$
\end{definition}

Using the $B_n$-action on $Q^n$, one interprets $Q$-colorings of the closure of a braid $\sigma\in B_n$ in terms of the points in $Q^n$ fixed by $\sigma$. A precise statement is given in the following lemma. A form of this statement appears as proposition $7.6$ in  \cite{ARTICLE:FennRourke1992}, with emphasis on coloring \emph{framed} links using \emph{racks} - a type of algebraic structure that generalizes quandles. For completeness we include a proof here, as well as an analogous statement for dominant colorings.
\begin{lemma}
	\label{Lma: fixed colorings}
	Let $Q\in\Qdl$ and let $n\in\NN$. Then for every $\sigma\in B_n$ there are natural bijections
	$$\Col_Q(\bclose\sigma)\simeq (Q^n)^{\beta(\sigma)}.$$
	$$\Col_Q^\dom(\bclose\sigma)\simeq (Q^{n,\dom})^{\beta(\sigma)}.$$
\end{lemma}

\begin{proof}
	Let $\sigma\in B_n$. A braid diagram for $\sigma$ consists of oriented arcs and crossings, much like the diagram of an oriented link - except that the overall strands have endpoints. We define $\Col_Q(\sigma)$, the set of $Q$-colorings of a braid diagram,  in much the same way as oriented link diagrams - with no conditions imposed at strand endpoints. We can describe $Q$-colorings of $\bclose\sigma$ in terms of $Q$-colorings of $\sigma$: Pulling back $c\in\Col_Q(\bclose\sigma)$ to $c_{|\sigma}\in\Col_Q(\sigma)$ gives an injective map
	$$\Col_Q(\bclose\sigma)\to\Col_Q(\sigma).$$
	Since the closure $\bclose\sigma$ is obtained by wrapping up the loose ends in a diagram of $\sigma$, crucially without introducing new crossings, those $c'\in\Col_Q(\sigma)$ residing in the image are precisely those with compatible colorings at endpoints: Let $p_0,p_1\colon \Col_Q(\sigma)\to Q^n$ denote the restrictions of $c'\in\Col_Q(\sigma)$ to the $n$ top (resp. bottom) endpoints of the braid. Then
	$$\Col_Q(\bclose\sigma)\iso\{c'\in\Col_Q(\sigma)\;|\;p_0(c')=p_1(c')\}.$$
	Any $c'\in\Col_Q(\sigma)$ is determined completely by the colors at the top, so 
	$$p_0\colon\Col_Q(\sigma)\iso Q^n$$
	is a bijection.
	Likewise colorings at the bottom, so
	$$p_1\colon\Col_Q(\sigma)\iso Q^n$$
	is also a bijection.
	Finally, note that colors at the top of $\sigma$ determine colors at the bottom of $\sigma$ by $p_1\circ p_0^{-1}=\beta(\sigma)$. This is because the $B_n$-action on $Q^n$ precisely emulates the progression from start to finish by legal $Q$-coloring:
	$$\begin{tikzpicture}
			\label{eqn: Brieskorn step}
			\draw (-4,0.5)node{$\sigma_i\;\;:$};
			\draw[->] (-2,1)node[anchor=south]{$x_{1}$}--(-2,0)node[anchor=north]{$x_{1}$};
			\draw (-1,0.5)node{$\cdots$};
			
			\draw[->] (1,1)node[anchor=south]{$x_{i+1}$}--(0,0)node[anchor=north]{$\qq{x_i}{x_{i+1}}$};
			\draw[-,line width=10pt, draw=white] (0,1)--(1,0);
			\draw[->] (0,1)node[anchor=south]{$x_{i}$}--(1,0)node[anchor=north]{$x_{i}$};
			
			\draw (2,0.5)node{$\cdots$};
			\draw[->] (3,1)node[anchor=south]{$x_{n}$}--(3,0)node[anchor=north]{$x_{n}$};
	\end{tikzpicture}$$
	Therefore
	$$\Col_Q(\bclose\sigma)\simeq \{c\in\Col_Q(\sigma)\;|\;p_0(c)=p_1(c)\} \xrightarrow[\sim]{p_0}\{\overrightarrow x\in Q^n\;|\;\overrightarrow x=\beta(\sigma)(\overrightarrow x)\}=\left(Q^n\right)^{\beta(\sigma)}.$$
	This also holds for all $R\le Q$. Therefore
	$$\Col_Q^\dom(\bclose\sigma)=\Col_Q(\bclose\sigma)\setminus\bigcup_{R<Q}\Col_R(\bclose\sigma)=(Q^{n})^{\beta(\sigma)}\setminus\bigcup_{R<Q}(R^{n})^{\beta(\sigma)}=$$
	$$=(Q^n\setminus\bigcup_{R<Q}R^n)^{\beta(\sigma)}=\left(Q^{n,\dom}\right)^{\beta(\sigma)}.$$
\end{proof}

\subsection{Statistics of quandle-colorings}

\begin{proposition}
	\label{Prp: c hat}
	Let $Q\in\Qdl^\fin$ and $n\in\NN$. Then the functions
	$$\col_{Q}(\bclose\sigma)\;,\; \col_{Q}^\dom(\bclose\sigma) \colon B_n\to\NN$$
	extend to locally-constant functions on $\widehat{B_n}$, the profinite completion of $B_n$:
	$$\begin{tikzcd}[column sep=large]
		B_n\ar[r,"\col_{Q}(\bclose\sigma)"]\ar[d]	&\NN & B_n\ar[r,"\col_{Q}^\dom(\bclose\sigma)"]\ar[d]	&\NN\\
		\widehat{B_n}\ar[ur,dashed, "c_{Q,n}"'] && \widehat{B_n}\ar[ur,dashed, "c^\dom_{Q,n}"']
	\end{tikzcd}$$
\end{proposition}

\begin{proof}
	By \cref{Lma: fixed colorings}, for every $\sigma\in B_n$,
	$$\Col_Q(\bclose\sigma)\simeq(Q^n)^{\beta(\sigma)}\;,\;\;\Col_Q^\dom(\bclose\sigma)\simeq(Q^{n,\dom})^{\beta(\sigma)}.$$
	 The functions $\col_{Q}(\bclose\sigma) $ and $\col_{Q}^\dom(\bclose\sigma)$ therefore factor through functions on $\Aut_\Set(Q^n)\;\;$ (resp. $\Aut_\Set(Q^{n,\dom})$ ) counting fixed points in $Q^n$ (resp. $Q^{n,\dom})$:
	 $$\col_{Q}(\bclose\sigma)\colon B_n\to \Aut_\Set(Q^n)\to\NN,$$
	 $$\col_{Q}^\dom(\bclose\sigma)\colon B_n\to \Aut_\Set(Q^{n,\dom})\to\NN.$$
	 Since $Q$ is finite, then $\Aut_\Set(Q^n),\Aut_\Set(Q^{n,\dom})\in\Grp^\fin$, thus
	 $$\col_Q(\bclose\sigma),\col_Q^\dom(\bclose\sigma)\colon B_n\to\NN$$
	 factor through finite quotients of $B_n$. Hence both $\col_{Q}(\bclose\sigma)$ and $\col_{Q}^\dom(\bclose\sigma)$ may be lifted to locally-constant functions on $\widehat{B_n}$.
\end{proof}

\begin{definition}
	\label{Dfn: dim Q}
	We denote the lifts described in \cref{Prp: c hat} of $\col_{Q}(\bclose\sigma)$ and $\col_{Q}^\dom(\bclose\sigma)$ to $\widehat{B_n}\to\NN$ by $c_{Q,n}$ and $c_{Q,n}^\dom$ respectively. When $n$ is clear from context, we abbreviate: 
	$$c_Q:=c_{Q,n}\;\;,\;\;\;\;c_Q^\dom:=c_{Q,n}^\dom.$$ As a compact Hausdorff group, we consider $\widehat{B_n}$ with its Haar measure $\mu$, and $c_Q,c_Q^\dom$ as random variables on $\widehat{B_n}$.
\end{definition}

We are ready to state the main results of this text:

\begin{theorem}
\label{Thm: Qdl HPoly}
	Let $Q\in\Qdl^\fin$. Then there exist integer-valued polynomials $P_Q,\;P_Q^\dom\in\QQ[x]$ s.t. 
	$$\int\limits_{\widehat{B_n}}c_Q d\mu=P_Q(n)\;\;\;\textrm{  and  }\;\;\; \int\limits_{\widehat{B_n}}c_Q^\dom d\mu=P_Q^\dom(n) $$
	for all $n\gg 0$.
\end{theorem}

The degrees of $P_Q(x)$ and $P_Q^\dom(x)$ are directly computed in terms of the quandle structure of $Q$. To that end we define:

\begin{definition}
	\label{Dfn: qdl dim}
	Let $Q\in\Qdl^\fin$. The \emph{dimension} of $Q$ is defined to be
	$$\dim_Q=\max_{R\le Q}\left|\pi_0(R)\right|.$$
\end{definition}

\begin{theorem}
	Let $Q\in\Qdl^\fin$. Then
	$$\deg P_Q^\dom=\left|\pi_0(Q)\right|-1,$$
	$$\deg P_Q=\dim_Q-1.$$
\end{theorem}

In order to prove this, we need the following first step:

\begin{proposition}
	\label{Prp: EQn orbit counting}
	Let $Q\in\Qdl^\fin$ and let $n\in\NN$. Then the expected values of $c_{Q,n}$ and $c_{Q,n}^\dom$ are equal to
	$$\int\limits_{\widehat{B_n}}c_Q d\mu =\left|Q^n/B_n\right|\;\;,\;\;\;\; \int\limits_{\widehat{B_n}}c_Q^\dom d\mu =\left|Q^{n,\dom}/B_n\right|.$$
\end{proposition}

\begin{proof}
Let $G\in\Grp^\fin$ be a finite group, acting on finite set $X$. By Burnside's lemma on orbit counting,
$$\frac1{|G|}\sum_{g\in G}\left|X^g\right|=\left|X/G\right|.$$
For infinite group $G$ acting on a finite set $X$, the following holds:
$$\int\limits_{\widehat G}\left|X^g\right|d\mu=|X/G|,$$
where $\widehat G$ is the profinite completion of $G$ with Haar measure $\mu$, and $X^g$ is defined for $g\in \widehat G$ via extension of $G\to\Aut_\Set(X)$ to $\widehat G$.
In particular for $G=B_n$ acting on $Q^n$ and $Q^{n,\dom}$ we have
$$\int\limits_{\widehat{B_n}}c_Q d\mu= \left|Q^n/B_n\right|\;\;\;\;\;\;\textrm{and}\;\;\;\;\;\;\int\limits_{\widehat{B_n}}c_Q^\dom d\mu= \left|Q^{n,\dom}/B_n\right|.$$
\end{proof}

\section{Quandles, Monoids and Algebras}
\subsection{Graded monoids and graded modules}

\begin{definition}
\label{Dfn: Setgr}
${}$\\
By a \emph{graded set} we mean a set $S$ graded over $\NN$:
	$$S=\coprod_{n=0}^\infty S_n.$$
	If $S$ is a graded set and $s\in S_n$, we denote by $|s|:=n$ the degree of $s$.
	By $\Setgr$ we denote the category of graded sets, where  morphisms are functions $f\colon S\to S'$ that respect the grading:
	$$\forall s\in S\;\;,\;\;\;\;|f(s)|=|s|.$$
	for all $n\in \NN$. A graded set $S$ is \emph{level-wise finite} if $S_n$ is finite for all $n\in\NN$.
\end{definition}

\begin{definition}
\label{Dfn: Mongr}
${}$\\
 A \emph{graded monoid} is an associative monoid $A$ with graded underlying set s.t.
	$$\forall a,b\in A\;\;,\;\;\;\;|a\cdot b|=|a|+|b|.$$
	By $\Mongr$ we denote the category of graded monoids, where morphisms are graded morphisms of monoids. There is an equivalence of categories
	$$\Mongr\cong\Mon_{/\NN},$$
	where $\Mon_{/\NN}$ is the category of monoids over $\NN:=(\NN,+)$.
\end{definition}

\begin{definition}
\label{Dfn: graded A-set}
${}$\\
Let $A\in\Mongr$. A \emph{graded $A$-set} is an $A$-set $S$ with grading s.t.
	$$\forall a\in A,\;s\in S\;\;,\;\;\;\;|a\cdot s|=|a|+|s|.$$
	By $\Modgr A$ we denote the category of graded $A$-sets, where morphisms are graded morphisms of $A$-sets.
	An $A$-set $S$ is \emph{finitely generated} if there exists a finite subset $T\subseteq S$ s.t.
	$$S=A\cdot T.$$
\end{definition}

\begin{definition}
	\label{Dfn: Enveloping Monoid}
	${}$\\
	Let $Q\in\Qdl$ be a quandle. We define the \emph{enveloping monoid}\footnote{In \cite{ARTICLE:KaMats2005} this is defined as the enveloping monoidal quandle.}
	of $Q$ to be the associative monoid given by the following generators and relations:
	$$\A_Q:=\NN\langle\; Q\;|\;\forall\; x,y\in Q\;,\;\;x\cdot y=\qq xy \cdot x\;\rangle.$$
	Note that $Q$ embeds canonically in $\A_Q$. This monoid was introduced in \cite[\textsection 3]{ARTICLE:KaMats2005} as the \emph{enveloping monoidal quandle} of $Q$ - but in the form of \cref{Prp: A_Q Braided Form}.
\end{definition}

\begin{proposition}
	Let $f\colon R\to Q$ be a morphism in $\Qdl$. Then there exists a unique well-defined multiplicative morphism
	$$\A_f\colon \A_R\to \A_Q$$
	extending $f\colon R\to Q$.
	The map $Q\mapsto \A_Q$ from $\Qdl$ to $\Mon$ is functorial.
\end{proposition}
\begin{proof}
	Such $\A_f$ is unique because $Q$ generates $\A_Q$. The composition of functions
	$$R\xrightarrow{f} Q\to \A_Q$$
	respects the relations defining $\A_R$:
	for every $x,y\in R$ we have
	$$f(x)\cdot f(y)=\qq{f(x)}{f(y)}\cdot f(x)=f(\qq xy)\cdot f(x)\in\A_Q.$$
	Therefore said function $R\to \A_Q$ extends to a multiplicative morphism
	$$\A_f\colon \A_Q\to \A_R.$$
	Functoriality is easily asserted.
\end{proof}

\begin{proposition}
	\label{Prp: A_Q Mongr Func}
	For every $Q\in\Qdl$, the monoid $\A_Q$ is naturally graded, making $Q\mapsto \A_Q$ into a functor
	$$\A\colon \Qdl\to\Mongr.$$
\end{proposition}
\begin{proof}
	We denote by $*\in \Qdl$ the terminal object - the quandle with one element. There is an equivalence of categories
	$$\Qdl\cong\Qdl_{/*}.$$
	The relations in defining $\A_*$ are redundant, therefore $\A_*$ is the free monoid on one element:
	$$\A_*=\NN\langle\;*\;\rangle\simeq\NN.$$
	Finally $\A$ defines a functor of over-categories
	$$\Qdl\iso\Qdl_{/*}\xrightarrow{\A}\Mon_{/\A_*}\simeq\Mon_{/\NN}\simeq\Mongr.$$
\end{proof}

\begin{proposition}
	\label{Prp: A_Q Braided Form}
	Let $Q\in\Qdl$. Then the graded monoidal morphism
	$$\coprod_nQ^n\simeq\NN\langle Q\rangle\to \A_Q$$
	defines an isomorphism of graded monoids
	$$\coprod_nQ^n/B_n\simeq \A_Q.$$
\end{proposition}
\begin{proof}
$$\A_Q=\NN\langle Q\;|\;\qq xy\cdot x= x\cdot y\;\;\;\;\forall x,y\in Q\rangle=$$
$$= \NN\langle Q\;|\; x_1\cdots{\qq{x_i}{x_{i+1}}\cdot x_i}\cdots x_n=
x_1\cdots {x_i\cdot x_{i+1}}\cdots x_n
\;\;,\;\forall x_1,\dots x_n\in Q,\;i=1\dots n-1\rangle= $$
$$=\NN\langle Q\;|\;\sigma_i(x)=x\;\;,\;\;\forall n,\;\forall x\in Q^n,\;\sigma_i\in B_n\rangle.$$
Since $B_n$ is generated by $\sigma_1\dots\sigma_{n-1}\in B_n$, we get for free that
$$\A_Q=\NN\langle Q\;|\;\sigma(x)=x\;\;\; \forall n\;,\;x\in Q^n,\;\sigma\in B_n\rangle.$$
This looks an awful lot like $\coprod Q^n/B_n$. We take care because the latter is obtained from $\NN\langle Q\rangle\simeq\coprod Q^n$ by taking various quotients as a \emph{set}. By showing that $\coprod Q^n/B_n$ carries the monoidal structure of $\coprod Q^n$, it will follow that
$$\coprod_n Q^n/B_n\iso \NN\langle Q\;|\;\sigma(x)=x\;\;\; \forall n\;,\;x\in Q^n,\;\sigma\in B_n\rangle = \A_Q .$$
Indeed, the product in $\NN\langle Q\rangle$ - which is given in graded components by $$Q^m\times Q^n\to Q^{m+n}$$ for all $m,n$ - is well-defined modulo actions of respective braid groups:
$$Q^m/B_m\times Q^n/B_n\to (Q^m\times Q^n)/(B_m\times B_n)\to Q^{m+n}/B_{m+n}$$
via the embedding
$$B_m\times B_n\hookrightarrow B_{m+n}.$$
As stated before, it follows that
$$\A_Q \simeq \coprod_n Q^n/B_n\in\Mongr.$$
\end{proof}

\begin{proposition}
	\label{Lma: sub R gives sub A_R}
	Let $Q,R\in\Qdl$ s.t. $R\le Q$. Then the induced morphism
	$$\A_R\to \A_Q$$
	in $\Mongr$ is an injection.
\end{proposition}
\begin{proof}
	For all $n\in \NN$, By \cref{Lma: Q^n is B_n equivar}, $R^n\subseteq Q^n$ is a sub-$B_n$-set. Passing to quotients, the map
	$$R^n/B_n\to Q^n/B_n$$
	is also an injection. Finally, from \cref{Prp: A_Q Braided Form} we have
	$$\A_R\simeq\coprod_nR^n/B_n\hookrightarrow \coprod_n Q^n/B_n\simeq \A_Q.$$
\end{proof}
\begin{remark}
	In light of \cref{Lma: sub R gives sub A_R}, as of now for any $Q\in\Qdl$ and sub-quandle $R$, we will identify $\A_R$ with a submonoid of $\A_Q$:
	$$R\le Q\;\Longrightarrow \;\A_R\le \A_Q.$$
\end{remark}

\begin{definition}
	\label{Dfn: A_Q dom}
	Let $Q\in\Qdl$. We define the graded set $\A_Q^\dom$ via
	$$\A_Q^\dom=\A_Q\setminus\bigcup_{R<Q}\A_R\in\Setgr.$$
\end{definition}

A simple inclusion-exclusion argument proves the following lemma:
\begin{lemma}
	\label{Lma: dom decomp A_Q}
	Let $Q\in\Qdl$. Then
	$$\A_Q=\coprod_{R\le Q}\A_R^\dom.$$
\end{lemma}

\begin{definition}
	Let $S$ be a graded, levelwise finite set. 
	If there exists an integer-valued 
	polynomial $P_S\in\QQ[x]$ s.t.
	$$|S_n|=P_S(n)\textrm{ for }n\gg 0,$$ 
	this (necessarily unique) polynomial is called the \emph{Hilbert polynomial of $S$}.
\end{definition}

\begin{example}
	Let $\alpha\in \NN$ and let $T_\alpha\in\Qdl$ be the trivial quandle with $\alpha$ elements, as in \cref{Ex: Triv Qdl}.
	Then $\A_{T_\alpha}$ is naturally isomorphic to the free commutative monoid generated by the set $T_\alpha $:
	$$\A_{T_\alpha}\simeq\NN\langle \;T_\alpha\;|\;\forall\;x,y\in T_\alpha\;,\;\;x\cdot y=\qq xy\cdot x=y\cdot x\;\rangle\simeq\NN[T_\alpha].$$
	In particular,
	$$\int\limits_{\widehat{B_n}} c_{T_\alpha}d\mu=|\A_{T_\alpha,n}|=|\Sym_n(T_\alpha)|={n+ \alpha-1\choose \alpha-1}$$ 
	for all $n\in\NN$, where ${x+ \alpha-1\choose \alpha-1}\in\QQ[x]$ 
	is a polynomial of degree $\alpha-1$.
\end{example}

\subsection{Monoids with Conjugation}

\begin{definition}
	A \emph{monoid with conjugation} is an associative monoid $A$ with a multiplicative homomorphism
	$$\varphi\colon A\to\Aut_\Mon(A)$$
	s.t. for all $a,b\in A$,
	$$a\cdot b=\varphi_a(b)\cdot a$$
	 - where $\varphi_a:=\varphi(a)\in\Aut_\Mon(A)$.
\end{definition}

\begin{lemma}
	\label{Lma: MonConj ambi div}
	Let $A$ be a monoid with conjugation. Then left and right divisibility in $A$ are equivalent.
\end{lemma}
\begin{proof}
	Let $c\in A$ and suppose $a,b\in A$ satisfy $a\cdot b=c$. Then
	$$\varphi_a(b)\cdot a =a\cdot b =b\cdot\varphi_b^{-1}(a)=c.$$
	Thus if $a$ divides $c$ from the left, then $a$ divides $c$ from the right, and if $b$ divides $c$ from the right then $b$ divides $c$ from the left.
\end{proof}

\begin{definition}
	Let $A$ be a monoid with conjugation. We write $a|b$ to denote that $a$ divides $b$, either from the left or the right - these are equivalent.
\end{definition}

\begin{lemma}
	\label{Lma: MonConj multidiv}
	Let $A$ be a monoid with conjugation. Let $a,b,c,d\in A$. 
	Then
	$$a|b\;,\;c|d\;\;\Longrightarrow \;\;a\cdot c|b\cdot d.$$
\end{lemma}
\begin{proof}
	Suppose $a|b$ and $c|d$. Let $x,y\in A$ s.t. $x\cdot a=b$ and $y\cdot c=d$. Then
	$$b\cdot d=x\cdot a\cdot y\cdot c= x\cdot \varphi_{a}(y)\cdot a\cdot c\;\;\;\;\Longrightarrow\;\;\;\; a\cdot c|b\cdot d\;.$$
\end{proof}

\begin{proposition}
	\label{Prop: graded action on monconj}
	Let $A$ be a monoid with conjugation. If $A$ is graded, then for all $a\in A$, $\varphi_a$ is a graded automorphism of $A$.
\end{proposition}
\begin{proof}
	Let $A$ be a graded monoid with conjugation and let $a\in A$. Then for all $b\in A$,
	$$|a|+|b|=|a\cdot b|=|\varphi_a(b)\cdot a| = |\varphi_a(b)|+|a|.$$
	Thus $\varphi_a\in\Aut_\Mongr(A)$ because $|\varphi_a(b)|=|b|$ for all $b\in A$.
\end{proof}

The following proposition is a rephrasing of \cite[Prop. 4]{ARTICLE:KaMats2005}:
\begin{proposition}
	\label{Prp: AQ is auto gen}
	Let $Q\in\Qdl^\fin$. Then $\A_Q$ is monoid with conjugation, with $\varphi\colon\A_Q\to\Aut_\Mongr(\A_Q)$ extending the quandle action $\varphi\colon Q\to\Aut_\Set(Q)$.
\end{proposition}
\begin{proof}
	By \cref{Prp: A_Q Mongr Func}, $\A\colon \Qdl\to\Mongr$ is a functor. Then for all $x\in Q$,
	$$Q\xrightarrow{\varphi}\Aut_\Qdl(Q)\xrightarrow{\A}\Aut_\Mongr(\A_Q)$$
	yields $\varphi_x\in\Aut_\Mongr(\A_Q)$ for all $x\in $. To show that this extends to a multiplicative map from all $\A_Q$, it suffices to show that for all $a\in\A_Q$,
	$$(\varphi_x\circ\varphi_y)(a)=(\varphi_{\qq xy}\circ\varphi_x)(a).$$
	This holds for all $a\in Q\simeq\A_{Q,1}$ (see \cref{Prp: Q to Inn(Q) is quandle morphism}),
	and since $\A_Q$ is generated by $Q$ as a monoid, the equality holds for all $a\in\A_Q$.
	
\end{proof}

\begin{proposition}
\label{Prp: Inn(A_Q)=Inn(Q)}
	Let $Q\in\Qdl^\fin$. Then the subgroup of $\Aut_\Mon(\A_Q)$ generated by $\{\varphi_x\;|\;x\in Q\}$ is naturally isomorphic to $\Inn(Q)$.
\end{proposition}
\begin{proof}
For all $x\in Q$, $\varphi_x\in\Aut(\A_Q)$ extends $\varphi_x\in\Aut(Q)$, so there is a restriction homomorphism
$$r\colon\langle\{\varphi_x\;|\;x\in\A_{Q,1}\simeq Q\}\rangle\to\Inn(Q).$$
The map $r$ is surjective because $\Inn(Q)$ is generated by $\{\varphi_x\;|\;x\in Q\}$. Since $\A_Q$ is generated as a monoid by $Q$, any $\psi\in\Aut_\Mon(\A_Q)$ is determined by its action on $Q$ - therefore $r$ is injective.
\end{proof}

\begin{proposition}
	\label{Prop: exp_A exists for auto-gen A}
	Let $Q\in\Qdl^\fin$. Then there exists $0<N\in\NN$ s.t. 
	$$\forall x\in Q\;,\;\;\varphi_x^N=\Id_{\A_Q}.$$
\end{proposition}
\begin{proof}
	Since $Q\in\Qdl^\fin$, $\Inn(Q)\in\Grp^\fin$.
	Therefore
	$$\varphi_x^{|\Inn(Q)|}=\Id_Q\in\Aut_\Qdl(Q).$$
	By \cref{Prp: Inn(A_Q)=Inn(Q)}, and taking $N=|\Inn(Q)|$,
	$$\varphi_x^N=\Id_{\A_Q}\in\Aut_\Mon(\A_Q)$$
	for all $x\in Q$.
\end{proof}

\begin{definition}
	\label{Dfn: auto gen exp}
	Let $Q\in\Qdl^\fin$.
	By $\exp_Q\in\NN$ we denote the smallest $N>0$ s.t. 
	$\varphi_x^N=\Id$ for all $x\in Q$.
	We denote by $J_Q\le \A_Q$ the sub-monoid of $\A_Q$ generated by 
	$x^{\exp_Q}$ for all ${x\in Q}$:
	$$J_Q=\langle \{x^{\exp_Q}\;|\;x\in Q\}\rangle_\Mon\le \A_Q.$$
\end{definition}

\begin{lemma}
	\label{Lma: JA Comm}
	Let $Q\in\Qdl^\fin$. Then $J_Q$ lies in the center of $\A_Q$: for all $x\in Q$, $x^{\exp_Q}$ commutes with all $\A_Q$.
\end{lemma}
\begin{proof}
	From \cref{Dfn: auto gen exp}, for all $x\in Q,a\in\A_Q$,
	$$\varphi_x^{\exp_Q}=\Id_{\A_Q}.$$
	Therefore
	$$x^{\exp_Q}\cdot a= \varphi_{x^{\exp_Q}}(a)\cdot x^{\exp_Q} =\varphi_x^{\exp_Q}(a)\cdot x^{\exp_Q}=a\cdot x^{\exp_Q}.$$
	Since $J_Q\le \A_Q$ is generated by these $x^{\exp_Q}$, $J_Q$ lies in the center of $\A_Q$.
\end{proof}

\begin{proposition}
\label{Prp: A_R saturated in A_Q}
	Let $Q\in\Qdl^\fin$ and let $R\le Q$. Let $a,b\in\A_Q$. Then
	$$ab\in\A_R\;\Longrightarrow\; a,b\in\A_R.$$
\end{proposition}
\begin{proof}
	 By \cref{Lma: Q^n is B_n equivar}, $R^n\subseteq Q^n$ is $B_n$-equivariant and by \cref{Prp: A_Q Braided Form},
	$$\A_{R,n}\simeq R^n/B_n.$$
	It follows that if $x\in Q$ and $c\in\A_R$ s.t. $x|c$, then $x\in R$. Now, suppose $a,b\in\A_Q$ are such that $ab\in\A_R$. write $a=a_1\cdots a_k$ and $b=b_1\cdots b_m$
	for $a_i,b_j\in Q$. Since $ab\in R$, by the above reasoning, $a_i,b_j$ must reside in $R$. Thus $a,b\in\A_R$.
\end{proof}

\begin{remark}
	For $Q\in\Qdl^\fin$, the submonoids $\A_R\le\A_Q$ are precisely the submonoids $B\le \A_Q$ satisfying the properties of \cref{Prp: A_R saturated in A_Q}. The notion of dominance therefore admits description in terms of monoids with conjugation. It is not clear whether \cref{Lma: dominant action new} holds in that generality, therefore the approach is of limited utility.
\end{remark}

\begin{lemma}
	\label{Lma: technical lemma new}
	Let $Q\in\Qdl^\fin$, let $N\in\NN$ and let $a\in \A_Q$ (resp. $a\in \A_Q^\dom$) s.t. $|a|>|Q|\cdot N$. Then there exist $x\in Q$ and $a'\in \A_Q$ (resp. $a'\in \A_Q^\dom$) s.t.
	$$a=x^{N}\cdot a'.$$
\end{lemma}
\begin{proof}
	Let $n=|a|$, and write $a=x_1\cdots x_n$ with $x_i\in Q$ for all $i$. Because $n>|Q|\cdot N$, there exists $x\in Q$ s.t. $x_i=x$ for at least $N+1$ values of $i$. In light of \cref{Lma: MonConj ambi div} and \cref{Lma: MonConj multidiv}, it follows that $x^{N+1}|a$. Write therefore $a=x^{N+1}\cdot b$ for some $b\in A$, and take $a':= x\cdot b$, s.t.
	$$a=x^N\cdot a'.$$
	We now show that if $a\in\A_Q^\dom$, then $a'\in\A_Q^\dom$. Let $R\le Q$ be a sub-quandle s.t. $a'=xb\in\A_R$. We show that $R=Q$. By \cref{Prp: A_R saturated in A_Q}, $x,b\in \A_R$. Therefore
	$$a=x^{N+1}b\in\A_R.$$
	But $a$ was assumed to be dominant, so $R=Q$ necessarily. Thus $a'\in\A_Q^\dom$.
\end{proof}

\begin{proposition}
	\label{Lma: A/B fin gen}
	Let $Q\in\Qdl^\fin$. Then $\A_Q$ and $\A_Q^\dom$ are finitely generated $J_Q$-sets.
\end{proposition}
\begin{proof}
	Firstly, $\A_Q$ is obviously a $J_Q$-set, for 
	$$J_Q\cdot \A_Q\subseteq \A_Q\cdot \A_Q=\A_Q.$$
	Furthermore $\A_Q^\dom$ is a $J_Q$-set: let $a\in J_Q$ and let $b\in \A_Q^\dom$. Assume $R\le Q$ is such that $ab\in\A_R$. Then $b\in\A_R$ from \cref{Prp: A_R saturated in A_Q}, implying that $R=Q$ because $b\in\A_Q^\dom$. It follows that $ab\in\A_Q^\dom$. Thus $\A_Q^\dom$ is a $J_Q$-set.
	Next, let
	$$S=\{a\in \A_Q\;|\;|a|\le |Q|\cdot \exp_Q\}\subseteq \A_Q.$$
	This set is finite as a union of finitely many $\A_{Q,n}$, which are themselves finite.
	We now show that $\A_Q=J_Q\cdot S$. Assuming the contrary, let  $a\in \A_Q\setminus J_Q\cdot S$ be of minimal degree $|a|=n$. Because $S\subseteq J_Q\cdot S$, we must have $a\notin S$, whereby
	$$n>|Q|\cdot \exp_Q.$$
	By \cref{Lma: technical lemma new} with $N=\exp_Q$,
	$$a=x^{\exp_Q}\cdot a'$$
	for some $x\in Q,\;a'\in \A_Q$. 
	Since $|a'|<|a|$ and $a\notin J_Q\cdot S$ is of minimal degree, 
	$$a'\in J_Q\cdot S.$$
	Now since $x^{\exp_Q}\in J_Q$, we have
	$$a=x^{\exp_Q}\cdot a'\in J_Q\cdot J_Q\cdot S=J_Q\cdot S,$$
	contrary to the assumption. Therefore
	$\A_Q=J_Q\cdot S$, with $S\subseteq \A_Q$ finite.
	
	Next, let $S'=S\cap \A_Q^\dom$. Then
	$$J_Q\cdot S'\subseteq J_Q\cdot \A_Q^\dom\subseteq \A_Q^\dom.$$
	Proving equality here is nearly identical to before: under the assumption that $\A_Q^\dom\setminus J_Q\cdot S'\neq\emptyset$, take $a\in \A_Q\setminus J_Q\cdot S'$ of minimal degree. By \cref{Lma: technical lemma new} we write 
	$$a=x^{\exp_Q}\cdot a'$$
	for some $x\in Q$ and $a'\in \A_Q^\dom$. As before we will again find that
	$$a=x^{\exp_Q}\cdot a'\in J_Q\cdot S',$$
	contrary to the assumption. 
	Therefore $\A_Q^\dom=J_Q\cdot S'$, with $S'$ finite.
\end{proof}

\begin{theorem}
	\label{Thm: auto-gen Hilbert Poly}
	Let $Q\in\Qdl^\fin$ and let $S\in\Modgr {\A_Q}$ be finitely-generated. Then there exists an integer-valued polynomial $P_S(x)\in\QQ[x]$ s.t. for all $n\gg 0$,
	$$\left|S_n\right|=P_S(n).$$
\end{theorem}

\begin{remark}
	The properties of $\A_Q$ that we use to proceed are that $\A_Q$ is a graded monoid with conjugation $A$, generated as a monoid by its finitely-many degree-$1$ elements. The insight gained by distilling these features is small, and we do not feel would justify the introduction of a new object and terminology.
\end{remark}

\subsection{From Monoids to Algebras}
Our choice of $\CC$ in this section is entirely arbitrary - any field will do.

\begin{definition}
	\label{Dfn: good algebra}
	A \emph{good} $\CC$-algebra consists of a graded, left Noetherian $\CC$-algebra $R$ and a finite set of generators $S\subseteq R_1$ of degree $1$, satisfying
	$$\forall s\in S\;,\;\;sR=Rs.$$
\end{definition}

\begin{proposition}
	\label{Prp: CA Noetherian}
	Let $Q\in\Qdl^\fin$. Then the algebra $\CC[\A_Q]$ is left Noetherian.
\end{proposition}
\begin{proof}
	The algebra $\CC[J_Q]$ is finitely-generated (by the set $\{x^{\exp_A}\;|\; x\in Q\}$) and by \cref{Lma: JA Comm}, it is commutative. Thus $\CC[J_Q]$ is Noetherian.
	Furthermore, it follows from \cref{Lma: A/B fin gen} that $\CC[\A_Q]$ is a finitely-generated left $\CC[J_Q]$-module. Putting the two together, $\CC[\A_Q]$ is itself left Noetherian.
\end{proof}

\begin{proposition}
	\label{Prp: auto gen to good alg}
	Let $Q\in\Qdl^\fin$. Then $(\CC[\A_Q],Q)$ is a good $\CC$-algebra.
\end{proposition}
\begin{proof}
	The grading on $\A_Q$ determines a grading on $\CC[\A_Q]$. $\CC[\A_Q]$ is generated as a $\CC$-algebra by the finite set of degree-$1$ monomials $Q\subseteq\CC[\A_Q]_1$.
	For all $x\in Q$, 
	the automorphisms $\varphi_x,\varphi_x^{-1}$ extend $\CC$-linearly to 
	automorphisms in $\Aut_{\textrm{Alg}_\CC}(\CC[\A_Q])$,
	also denoted by $\varphi_x,\varphi_x^{-1}$.
	The equalities
	\begin{equation}
		\label{Eqn: CAQ is good} x\cdot r=\varphi_x(r)\cdot x\;\;\;\;,\;\;\;\;\;\;
	r\cdot x=x\cdot \varphi_x^{-1}(r)
	\end{equation}
	are both $\CC$-linear in $r$ and hold for all $r\in \A_Q\subseteq\CC[\A_Q]$. Hence (\ref{Eqn: CAQ is good}) holds for all $x\in Q$ and $r\in\CC[\A_Q]$.
	Together these imply that
	$$x\CC[\A_Q]= \CC[\A_Q]x$$ 
	for all $x\in Q$. 
	Finally, from \cref{Prp: CA Noetherian} $\CC[\A_Q]$ is 
	left Noetherian. 
	Thus $(\CC[\A_Q],Q)$ is good. 
\end{proof}

\begin{proposition}
	Let $Q\in\Qdl^\fin$. Then $\CC[\A_Q^\dom]$ is a finitely-generated, graded left $\CC[\A_Q]$-module.
\end{proposition}
\begin{proof}
	From \cref{Lma: A/B fin gen}, $\A_Q^\dom$ is a finitely-generated $J_Q$-set, therefore a finitely-generated graded $\A_Q$-set. Thus $\CC[\A_Q^\dom]$ is a finitely-generated graded left $\CC[\A_Q]$-module.
\end{proof}

\begin{lemma}
	\label{Lma: good alg induction}
	Let $(R,S\subseteq R_1)$ be a good $\CC$-algebra
	s.t. $S\neq\emptyset$. 
	Fix some $s_0\in S$. 
	Let $R':=R/Rs_0$. 
	For every $r\in R$, let $\overline r$ denote $\overline r=r+Rs_0\in R'$. 
	Let $S'=\{\overline{s}\in R'\;|\;s_0\neq s\in S\}$. 
	Then 
	$(R',\;S')$ is a good $\CC$-algebra, with $|S'|<|S|$.
\end{lemma}
\begin{proof}
	Since $(R,S)$ is good and $s_0\in S\subseteq R_1$, the set $Rs_0$ is a two-sided, homogeneous ideal.
	Therefore $R':=R/s_0R$ is a graded $\CC$-algebra 
	with grading induced from $R$. 
	The set $\overline S=\{\overline s\;|\;s\in S\}\subseteq R'_1$ generates $R'$, 
	and also contains $\overline{s_0}=0_{R'}$. Therefore the set
	$$S'= \{\overline s\;|\;s_0\neq s\in S\}\subseteq R'_1$$
	also generates $R'$. Moreover, for any $s'=\overline s\in S'$, 
	$$s'R'=\overline s\overline R=\overline{sR}=\overline{Rs}=\overline R\overline s=R's'.$$ 
	Finally, $R'$ is left Noetherian as a quotient of $R$. 
	We conclude that $(R',S')$ is good. Note that $|S'|\le |S\setminus\{s_0\}|=|S|-1$.
\end{proof}
The following lemma is fairly standard, so we omit the proof.

\begin{lemma}
	\label{Lma: discrete diff}
	Let $f\colon\NN\to\ZZ$ be a function. Then $f$ eventually coincides with a polynomoial of degree $\le d$ in $\QQ[x]$ iff $\delta(f)(n)=f(n)-f(n-1)$ eventually coincides with a polynomial of degree $\le d-1$ in $\QQ[x]$.
\end{lemma}

\begin{proposition}
	\label{Prp: good HPoly}
	Let $(R,S\subseteq R_1)$ be a good $\CC$-algebra. 
	Let $M$ be a finitely-generated graded left $R$-module. 
	Then there exists an integer-valued polynomial $P_M(x)\in\QQ[x]$ of degree $\deg P_M <|S|$ s.t.
	$$\dim_\CC M_n=P_M(n)\;\textrm{ for }\;n\gg 0.$$ 
\end{proposition}
\begin{proof}
	We denote by $H_M(n)=\dim_\CC M_n$. The proof is by induction on $|S|$: \\
	If $|S|=0$, then $R=\CC$ and $M$ is a finitely-generated graded $\CC$-vector space. Therefore $M_n=0$ for $n\gg 0$, and $H_M$ eventually coincides with the zero polynomial which has degree $<0$.\\
	If $|S|>0$, let $s_0\in S$ be any element. 
	Let $R':=R/s_0R$, and let 
	$S':=\{\overline s\in R'\;|\;s_0\neq s\in S\}$. 
	By \cref{Lma: good alg induction}, $(R',S')$ is good.
	Consider the exact sequence 
	\begin{equation}
	\label{eqn: s_0 ExSeq skew-comm}
	0\to K\to M[-1]\stackrel{s_0\cdot} {\longrightarrow} M\to M/s_0M\to 0.
	\end{equation}
	The middle map is a graded morphism of graded $\CC$-vector spaces, so both kernel $K$ and cokernel $M/s_0M$ are graded $\CC$-vector spaces. Moreover, they admit natural graded left $R$-module structures:
	Let $r\in R$ and let $k\in K$. Since $s_0R=Rs_0$, take $r_1\in R$ s.t. $s_0r=r_1s_0$. Then
	$$s_0\cdot k=0\;\Longrightarrow\;s_0\cdot rk=r_1\cdot s_0k=r_1\cdot 0=0,$$
	therefore $rk\in K$. Thus $K$ is a left $R$-module. Let $r\in R$ and let $s_0m\in s_0M$. Since $s_0R=Rs_0$, take $r_2\in R$ s.t. $rs_0=s_0r_2$. Then
	$$r\cdot s_0m=s_0\cdot r_2m\in s_0M.$$
	Thus $s_0M$ is a left $R$-module, and so is the quotient $M/s_0M$. These left $R$-module structures are compatible with the grading on $K$ and $M/s_0M$.
	The left $R$-module $M$ is Noetherian ($R$ is left-Noetherian and $M$ is finitely generated), thus all sub-quotients of $M$ are finitely-generated. In particular $K$ and $M/s_0M$ are finitely-generated. Furthermore, both $K$ and $M/s_0M$ are annihilated by $s_0\in R$, making them finitely-generated, graded $R'=R/Rs_0$-modules.
	
	Since $(R',S')$ is good and $|S'|<|S|$, 
	by induction for $K$ and $M/s_0M $,
	there exist integer-valued $P_K,P_{M/s_0M}\in\QQ[x]$ with $\deg P_K,\deg P_{M/s_0M}\le |S'|-1$ s.t.
	$$ \dim_\CC K_n=P_K(n)\;,\;\;\dim_\CC(M/s_0M)_n=P_{M/s_0M}(n)\;\;\;\;\forall\;n\gg0.$$
	Moreover, from levelwise exactness in 
	(\ref{eqn: s_0 ExSeq skew-comm}) we find for every $n$ 
	the equality 
	$$\dim_\CC M_n-\dim_\CC M_{n-1} = \dim_\CC (M/s_0M)_n - \dim_\CC K_n.$$
	Therefore for all $n\gg 0$,
	$$\dim_\CC M_n-\dim_\CC M_{n-1} = P_{M/s_0M}(n) - P_K(n)= (P_{M/s_0M} - P_K)(n),$$
	where
	$$\deg (P_{M/s_0M} - P_K) \le |S'|-1.$$
	By \cref{Lma: discrete diff}, There is some integer-valued polynomial $P_M\in\QQ[x]$ of degree $\deg P_M\le |S'|<|S|$ s.t. for all $n\gg 0$,
	$$\dim_\CC M_n=P_M(n).$$
	By induction therefore, for all good $(R,S)$ and all graded, finitely-generated left $R$-modules $M$ there exists $P_M\in\QQ[x]$, integer-valued with $\deg P_M<|S|$ s.t.
	$$\dim_\CC M_n=P_M(n)\;\textrm{ for }\;n\gg 0.$$
\end{proof}

\begin{proof}(\cref{Thm: auto-gen Hilbert Poly})\\
	Let $Q\in\Qdl^\fin$ and let $S$ be a finitely-generated graded $\A_Q$-set. By \cref{Prp: auto gen to good alg}, $(\CC[\A_Q],Q)$ is a good $\CC$-algebra. Since $S$ is a finitely-generated graded $\A_Q$-set, $\CC[S]$ is a finitely-generated graded $\CC[\A_Q]$-module. By \cref{Prp: good HPoly} there is an integer-valued polynomial $P_S\in \QQ[x]$ s.t.
	$$|S_n|=\dim_\CC\CC[S]_n=P_S(n)$$
	for all $n\gg 0$.
\end{proof}

Now we are ready to prove the first main result.
\begin{proof}(\cref{Thm: Qdl HPoly}) 
	Let $Q\in\Qdl^\fin$.
	By \cref{Prp: EQn orbit counting}
	and \cref{Prp: A_Q Braided Form}, 
	$$\int\limits_{\widehat{B_n}}c_{Q}d\mu=|Q^n/B_n|=|\A_{Q,n}|\;\;,\;\;\;\; \int\limits_{\widehat{B_n}}c^\dom_{Q}d\mu=|Q^{n,\dom}/B_n|=|\A_{Q,n}^\dom|.$$
	By \cref{Thm: auto-gen Hilbert Poly},
	there are integer-valued polynomials $P_{\A_Q},P_{\A_Q^\dom}\in\QQ[x]$ s.t.
	$$\int\limits_{\widehat{B_n}}c_{Q}d\mu=P_{\A_Q}(n)\;\;,\;\;\;\; \int\limits_{\widehat{B_n}}c^\dom_{Q}d\mu=P_{\A_Q^\dom}(n)$$
	for all $n\gg 0$.
\end{proof}

\begin{definition}
\label{Dfn: dom poly}
	Let $Q\in\Qdl^\fin$. We denote the polynomials of \cref{Thm: Qdl HPoly}
	$$P_Q:=P_{\A_Q}\in\QQ[x]\;\;,\;\;\;\; P_Q^\dom:=P_{\A_Q^\dom}\in\QQ[x].$$
	The polynomial $P_Q$ is the \emph{Hilbert polynomial} of $Q$.
\end{definition}

\begin{remark}\label{remEVW}
	The premise in \cite{WEBSITE:ElVenWest2015} is $(G,c)$, where $G$ is a finite group and $c\subseteq G$ is a conjugacy class satisfying a certain non-splitting principle (\cite[Def 3.1]{WEBSITE:ElVenWest2015}). This $c$ with conjugation is a quandle. The Hurwitz spaces are then given as homtopy quotients
	$$\textrm{Hur}_{G,n}\simeq G^n\sslash B_n\;,\;\;\textrm{Hur}_{G,n}^c\simeq c^n\sslash B_n.$$
	The ring $R$ in \cite[section 3]{WEBSITE:ElVenWest2015} for $k=\CC$ is therefore
	$$R=\bigoplus_n H_0(\textrm{Hur}_{G,n}^c,\CC)\simeq \bigoplus_n\CC[\pi_0(\textrm{Hur}_{G,n}^c)]\simeq\bigoplus_n\CC[c^n/B_n]\simeq\CC[\A_c],$$
	and \cite[Thm 6.1]{WEBSITE:ElVenWest2015} implies that for all $i\in\NN$, the graded, left $R$-module $\displaystyle{M_i:=\bigoplus_n H_i(\textrm{Hur}_{G,n}^c,\CC)}$ is finitely-generated. \Cref{Prp: good HPoly} then implies that for all $i\in\NN$, the stable range of $\dim_\CC H_i(\textrm{Hur}_{G,n}^c,\CC)$ of \cite[Thm 6.1]{WEBSITE:ElVenWest2015} is in fact constant in $n$, for $n$ sufficiently large. Also implied - albeit in a roundabout manner by \cref{Thm: dim deg 1} - is the fact that $\dim_c=1$. Our choice of field $k=\CC$ is  arbitrary - so this statement will hold for any $k$ for which $|\Inn(Q)|\in k^\times$.
\end{remark}

\section{Calculating $\deg P_Q$ and $\deg P_Q^\dom$.}

The degree of $P_Q$, the Hilbert polynomial of $Q\in\Qdl^\fin$ 
can be expressed in terms of all sub-quandles of $Q$.

\begin{proposition}
	\label{Prp: dom sum}
	Let $Q\in\Qdl^\fin$. Then
	$$P_Q=\sum_{R\le Q} P_{R}^\dom.$$
\end{proposition}
\begin{proof}
	By \cref{Lma: dom decomp A_Q} we have the following graded partition of $\A_Q$:
	$$\A_Q=\coprod_{R\le Q}\A_R^\dom.$$
	Thus for all $n\in\NN$,
	\begin{equation}\label{eqn: fin sum dom}|\A_{Q,n}|=\sum_{R\le Q}|\A_{R,n}^\dom|.\end{equation}
	By \cref{Dfn: dom poly}, for all $R\le Q$,
	$$|\A_{R,n}^\dom|=P_R^\dom(n)$$
	for $n\gg 0$, the threshold depending on $R$. The sum in (\ref{eqn: fin sum dom}) runs over finitely-many sub-quandles $R\le Q$, therefore there is a common threshold, and
	$$P_Q(n)=|\A_{Q,n}|=\sum_{R\le Q}|\A_{R,n}^\dom| =
	\sum_{R\le Q}P_R^\dom(n)$$
	for $n\gg 0$. Since $P_Q,\;\displaystyle{\sum_{R\le Q} P_{R}^\dom}$ coincide for input $n\gg 0$, they must be equal:
	
	$$ P_Q=\sum_{R\le Q}P_R^\dom\in\QQ[x].$$
\end{proof}

We proceed to calculate $\deg P_Q^\dom$.

\begin{lemma}
	\label{Lma: dominant action new}
	Let $Q\in\Qdl^\fin$ and let $a\in\A_Q^\dom$. Then
	$$\langle\{\varphi_b\;|\;b\in\A_Q,\;b|a\}\rangle_\Grp=\Inn(Q).$$
\end{lemma}
\begin{proof}
	Denote by $G_a:= \langle\{\varphi_b\;|\;b\in\A_Q,\;b|a\}\rangle\le \Inn(Q)$ and let $S\subseteq Q$ denote
	$$S_a=\{x\in Q\;|\;\varphi_x\in G_a\}.$$
	In fact $S\le Q$ is a sub-quandle: Let $x,y\in S_a$. Then by \cref{Prp: Q to Inn(Q) is quandle morphism}
	$$\varphi_x,\varphi_y\in G_a\;\Longrightarrow\; \varphi_{\qq xy}=\varphi_x\varphi_y\varphi_x^{-1} \in G_a \;\Longrightarrow\;\qq xy\in S_a.$$
	Therefore $S_a\le Q$ is a sub-quandle.
	Write $a=x_1\dots x_n$ for $x_i\in \A_{Q,1}\simeq Q$.
	For all $i$, $\varphi_{x_i}\in G_a$, therefore $x_i\in S_a$, hence
	$$a=x_1\cdots x_n\in \A_{S_a}.$$
	But $a\in\A_Q^\dom$, it follows that $S_a=Q$. Therefore $\varphi_x\in G_a$ for all $x\in Q$. Thus
	$$G_a\le \langle\{\varphi_x\;|\;x\in Q\}\rangle_\Grp=\Inn(Q),$$
	and
	$$\langle\{\varphi_b\;|\;b\in\A_Q,\;b|a\}\rangle =\Inn(Q).$$

\end{proof}

\begin{proposition}
	\label{Lma: acrobatics}
	Let $Q\in\Qdl^\fin$ and let $a\in\A_Q^\dom$. Then for all $x\in Q$ and all $\psi\in\Inn(Q)$,
	$$x^{\exp_Q}\cdot a=\psi(x)^{\exp_Q}\cdot a.$$
\end{proposition}
\begin{proof}
	Let $b\in\A_Q$ s.t. $b|a$, and write $a=b\cdot c$ for some $c\in\A_Q$. By \cref{Lma: JA Comm}, for all $x\in Q$,
	$$x^ {\exp_Q}\cdot a=x^ {\exp_Q}\cdot b\cdot c=b\cdot x^ {\exp_Q}\cdot c=
	\varphi_b(x^ {\exp_Q})\cdot b\cdot c = \varphi_b(x)^ {\exp_Q}\cdot a.$$
	By \cref{Lma: dominant action new}, $\Inn(Q)$ is generated as a group by $\varphi_b$ for all $b|a$ in $\A_Q$. Thus
	$$x^{\exp_Q}\cdot a=\psi(x)^{\exp_Q}\cdot a$$
	for all $x\in Q$ and all $\psi\in\Inn(Q)$.
\end{proof}

\begin{lemma}
	\label{Lma: inflation in deg calc}
	Let $A$ be a commutative graded monoid, freely-generated by $k$ elements of degree $N$. Let $C$ be a finitely-generated graded $A$-set that admits a Hilbert polynomial $P_C\in\QQ[x]$. Assume that there exists $c\in C$ s.t. $|c|\in N\ZZ$ and
	$$A\xrightarrow{\cdot c} C$$ is injective. Then
	$$\deg P_C=k-1.$$
\end{lemma}
\begin{proof}
	Let $S\subseteq C$ be a finite set of generators. Define a graded monoid $A'$ and graded $A'$-set $C'$ as follows:
	$$A'_n=A_{nN}\;\;\;\;,\;\;\;\;\;\;\;\;C'_n=C_{nN}.$$
	Then for all $n\gg 0$,
	$$|C'_n|=|C_{nN}|=P_C(nN).$$
	Therefore $C'$ admits a Hilbert polynomial $P_{C'}$ satisfying
	$$\deg P_C=\deg P_{C'}.$$
	The $\CC$-algebra $\CC[A']$ is graded, abelian, freely generated by $k$ elements of degree $1$. Consider $\CC[C']$, a graded $\CC[A']$-module. On one hand $\CC[C']$ is finitely-generated, therefore $\deg P_{C'}\le k-1$. On the other hand $\CC[C']$ contains a free sub-module, namely $c\CC[A']$, whereby $\deg P_{C'}\ge k-1$. Altogether,
	$$\deg P_C=\deg P_{C'}=k-1.$$
\end{proof}

\begin{proposition}
	\label{Prp: dim dom}
	Let $Q\in\Qdl^\fin$. Then
	$$\deg P_Q^\dom=|\pi_0(Q)|-1.$$
\end{proposition}
\begin{proof}
	Consider the quotient map
	$$p\colon Q\to\pi_0(Q)=Q/\Inn(Q).$$
	Let $\m\subseteq Q$ be a set of representatives in $\pi_0(Q)$, synonymous with a section 
	$$\m \colon \pi_0(Q)\to Q$$
	in $\Set$ of the aforementioned quotient. For every $x\in Q$ let $\psi_{\m,x}\in\Inn(Q)$ s.t.
	$$\psi_{\m,x}(x)\in\m.$$
	We define the submonoid $J_\m\le \A_Q$ via
	$$J_\m:=\langle\{x^{\exp_Q}\;|\;x\in\m \}\rangle_\Mon\le J_Q\le \A_Q.$$
	The monoid $J_\m$ is free commutative: By \cref{Lma: JA Comm}, for all $x\in Q$, $x^{\exp_Q}$ commutes with $\A_Q$. Consider the composition
	\begin{equation}\begin{tikzcd}[row sep=tiny]
	\NN[\pi_0(Q)]\ar[r]&J_\m\ar[r,hook,"\subseteq"]&\A_Q\ar[r,two heads]&\NN[\pi_0(Q)]\\
	z\ar[r,mapsto]&\m(z)^{\exp_Q}&x\ar[r,mapsto]&p(x)
	\label{Eqn: expansion}\end{tikzcd}\end{equation}
	The map $\NN\langle Q\rangle\to\NN[\pi_0(Q)]$ is well-defined  on $\A_Q$ since for all $x,y\in Q$,
	$$\qq xy\cdot x=\varphi_x(y)\cdot x=y\cdot x=x\cdot y\in\NN[Q/\Inn(Q)]=\NN[\pi_0(Q)].$$
	Since $\NN[\pi_0(Q)]$ is free commutative, the composition in (\ref{Eqn: expansion}) is injective:
	$$z\longmapsto \m(z)^{\exp_Q}\longmapsto (p\circ\m)(z)^{\exp_Q}=z^{\exp_Q},$$
	therefore $J_\m$ is a commutative monoid generated freely by $\{x^{\exp_Q}\;|\;x\in\m\}$.
	By \cref{Lma: A/B fin gen}, there is some finite $S\subseteq \A_Q^\dom$ s.t.
	$$A_Q^\dom=J_Q\cdot S.$$
	Let $a\in\A_Q^\dom$ be any element. Write
	$$a=\prod_{x\in Q}x^{\exp_Qn_x}\cdot s.$$
	Since the $x^{\exp_Q}$ commute and $s$ is dominant, from \cref{Lma: acrobatics} we have
	$$a=\prod_{x\in Q}x^{\exp_Qn_x}\cdot s= \prod_{x\in Q}\psi_{\m,x}(x)^{\exp_Qn_x}\cdot s\in J_\m\cdot S.$$
	Thus $\A_Q^\dom$ is a finitely-generated $J_\m$-set.
	Next, consider
	$$c:=\prod_{x\in Q}x^{\exp_Q}\in\A_Q^\dom\;\;,\;\;\;\;
	|c|=|Q|\exp_Q.$$
	The map
	$$J_\m\xrightarrow{\cdot c} \A_Q^\dom$$
	Is injective. Also $J_\m$ is a graded, commutative monoid freely generated by $|\m|=\pi_0(Q)$ elements of degree $\exp_Q$. Therefore by \cref{Lma: inflation in deg calc},
	$$\deg P_{\A_Q^\dom}=\deg P_Q^\dom=|\pi_0(Q)|-1$$
\end{proof}

\begin{theorem}
	\label{Thm: dim deg 1}
	Let $Q\in\Qdl^\fin$. Then 
	$$\deg P_Q=\dim_Q-1.$$
\end{theorem}
\begin{proof}
	By \cref{Prp: dom sum},
	$$\displaystyle{P_Q=\sum_{R\le Q}P_R^\dom}.$$
	For all $R\le Q$ unless $P_R=0$, the leading coefficient in $P_R^\dom$ is positive. Thus
	$$\deg P_Q=\deg\sum_{R\le Q}P_R^\dom= \max_{R\le Q}\left(\deg P_R^\dom\right).$$
	By \cref{Prp: dim dom}, for all $R\le Q$,
	$$\deg P_R^\dom=|\pi_0(R)|-1.$$
	Therefore
	$$\deg P_Q= \max_{R\le Q}\left(\deg P_R^\dom\right)=\max_{R\le Q}\left(|\pi_0(R)|-1\right)=\dim_Q-1.$$
\end{proof}

\section{Products in $\Qdl$ and Refined Statistics}

\begin{definition}
	\label{Dfn: qdl prod}
	Let $Q,R\in\Qdl$. 
	The set $Q\times R$ has a natural quandle structure 
	$$\qq{(x,y)}{(x',y')}=(\qq x{x'},\qq y{y'}).$$
\end{definition}

\begin{proposition}
\label{Prp: qdl prod}
	The quandle product of \cref{Dfn: qdl prod} is the categorical direct product in $\Qdl$.
\end{proposition}
\begin{proof}
	Let $Q,R,S\in\Qdl$. Then
	$$\Hom_\Qdl(S,Q\times R_{\;(\ref{Dfn: qdl prod})})=\{f\colon S\to Q\times R\;|\;f(\qq xy)=\qq{f(x)}{f(y)}\}\simeq$$
	$$\simeq\{(f_1\colon S\to Q,f_2\colon S\to R)\;|\;f_1(\qq xy)=\qq{f_1(x)}{f_1(y)}\;,\;\; f_2(\qq xy)=\qq{f_2(x)}{f_2(y)}\}=$$
	$$=\Hom_\Qdl(S,Q)\times\Hom_\Qdl(S,R).$$
\end{proof}

\begin{proposition}
	\label{Prp: fib sub}
	Let $f\colon R\to Q$ be a morphism in $\Qdl$. Let $z\in Q$. Then
	$$f^{-1}(z)\le R.$$
\end{proposition}
\begin{proof}
Let $x,y\in f^{-1}(z)$. Then $f(x)=f(y)=z$, and
$$f(\qq xy)=\qq{f(x)}{f(y)}=\qq zz=z\;\Longrightarrow\; \qq xy\in f^{-1}(z).$$
Thus
$$f^{-1}(z)\le R.$$
\end{proof}

\begin{proposition}
	\label{Nota: (phi,r)}
	\label{Lma: (phi,r)}
	\label{Prp: surj Inn}
	Let $f\colon R\to Q$ be a surjective morphism in $\Qdl$. Then $f$ induces a surjective group homomorphism
	$$f_*\colon \Inn(R)\twoheadrightarrow\Inn(Q)$$
	s.t. for all $\psi\in\Inn(R)$,
	$$f\circ\psi=f_*(\psi)\circ f.$$
\end{proposition}
\begin{proof}

For $Q\in\Qdl$, one can factor the map $\varphi\colon Q\to \Inn(Q)$ through $F_Q$, the free group on $Q$, thus extending the $Q$-action on itself to an $F_Q$-action on $Q$:
	$$\begin{tikzcd}
		Q\ar[r] &F_Q\ar[r,"\varphi"] &\Inn(Q)
	\end{tikzcd}.$$
	Now let $f\colon R\to Q$ be a surjective quandle morphism. Consider the diagram
	$$\begin{tikzcd}
	R\ar[r]\ar[d,"f"']\ar[dr,"(*)",phantom] &F_R\ar[r,"\varphi",two heads]\ar[d,dashed] \ar[dr,"(**)",phantom] &\Inn(R)\ar[d,dashed]\\
	Q\ar[r] &F_Q\ar[r,"\varphi"',two heads] &\Inn(Q)
\end{tikzcd}$$
	The map $R\xrightarrow {\;f\;} Q$ extends to a homomorphism between free groups $F_R\to F_Q$ - also denoted by $f$ - closing up square $(*)$ above. For all $x,y\in R$ we have
	$$(f\circ \varphi_x)(y)=f(\qq xy)=\qq{f(x)}{f(y)}=(\varphi_{f(x)}\circ f)(y).$$
	Thus for all $x\in R$ there is an equality of functions $R\to Q$ -
	\begin{equation}
	\label{square dancing}
	\forall x\in R\;,\;\;f\circ\varphi_x=\varphi_{f(x)} \circ f.
	\end{equation}

	If $w,w'\in F_R$ are words s.t.
	$$f\circ \varphi_w=\varphi_{f(w)}\circ f\;,\;\; f\circ \varphi_{w'}=\varphi_{f(w')}\circ f,$$
	then
	$$f\circ \varphi_{ww'}= f\circ \varphi_{w}\circ\varphi_{w'} =\varphi_{f(w)}\circ f\circ\varphi_{w'}=\varphi_{f(w)}\circ\varphi_{f(w')}\circ f=$$
	$$=\varphi_{f(w)f(w')}\circ f=\varphi_{f(ww')}\circ f$$
	and
	$$\varphi_{f(w^{-1})}\circ f=\varphi_{f(w)^{-1}}\circ f\circ \varphi_w\circ (\varphi_w)^{-1}=(\varphi_{f(w)})^{-1}\circ \varphi_{f(w)}\circ f\circ \varphi_{w^{-1}}=f\circ\varphi_{w^{-1}}.$$
	Thus the equality in (\ref{square dancing}) extends to all $F_R$:
	\begin{equation}
	\label{free square dancing}
	\forall w\in F_R\;,\;\;f\circ\varphi_w=\varphi_{f(w)} \circ f.
	\end{equation}
	Suppose $w\in F_R$ is such that $\varphi_w=\Id_R\in\Inn(R)$. Then for all $y\in R$,
	$$\varphi_{f(w)}(f(y))=(\varphi_{f(w)}\circ f)(y)=(f\circ \varphi_w)(y)=f(y).$$
	Recall that $f$ is surjective, thus $\varphi_{f(w)}$ acts trivially on $f(R)=Q$. In other words $f$ maps $\ker(F_R\to \Inn(R))$ to $\ker(F_Q\to \Inn(Q))$, hence there is a group homomorphism
	$$f_*\colon \Inn(R)\to\Inn(Q)$$
	completing the square $(**)$ above.
	Since the maps $F_R\xrightarrow{\;f\;}F_Q\xrightarrow{\;\varphi\;}\Inn(Q)$ are both surjective, so is $f_*$. Furthermore, the action of $F_R$ on $R$ (resp. $F_Q$ on $Q$) factors through $F_R\xrightarrow{\;\varphi\;}\Inn(R)$ (resp. $F_Q\xrightarrow{\;\varphi\;}\Inn(Q)$), therefore the equality in (\ref{free square dancing}) descends to $\Inn(R)$:
	$$\forall \psi\in \Inn(R)\;,\;\;f\circ\psi=f_*(\psi) \circ f.$$
\end{proof}

\begin{corollary}
	Let $f\colon R\to Q$ be a surjective morphism of quandles, and let $x,y\in Q$ be such that $[x]_\sim=[y]_\sim\in\pi_0(Q)$. Then there exists $\psi\in\Inn(R)$ s.t. 
	$$\psi\colon f^{-1}(x)\iso f^{-1}(y)$$
	is an isomorphism of quandles.
\end{corollary}
\begin{proof}
	Since $[x]_\sim=[y]_\sim\in \pi_0(Q)$, let $\psi_0\in\Inn(Q)$ be s.t.
	$$\psi_0(x)=y.$$
	Since $f$ is surjective, by \cref{Prp: surj Inn} we can lift $\psi_0$ to some $\psi\in\Inn(R)$ s.t. 
	$$f\circ \psi=\psi_0\circ f,$$
	which also implies that
	$$\psi_0^{-1}\circ f=f \circ\psi^{-1}.$$
	This $\psi$ maps $f^{-1}(x)$ bijectively onto $f^{-1}(y)$: Let $z\in f^{-1}(x)$. Then
	$$ f(\psi(z))=\psi_0(f(z))=\psi_0(x)=y\;\;\;\Longrightarrow\;\;\; \psi (z)\in f^{-1}(y).$$
	For $z\in f^{-1}(y)$,
	$$f(\psi^{-1}(z))=\psi_0^{-1}(f(z))=\psi_0^{-1}(y)=x\;\;\;\Longrightarrow\;\;\;\psi^{-1}(z)\in f^{-1}(x).$$
	By \cref{Prp: fib sub}, $f^{-1}(x),f^{-1}(y)$ are sub-quandles of $R$, and since $\psi\in\Inn(R)$
	 is a quandle automorphism of $Q$, it induces a quandle isomorphism
	$$\psi\colon f^{-1}(x)\iso f^{-1}(y).$$
\end{proof}

\begin{lemma}
	\label{Lma: pi_0 fibration bound}
	Let $f\colon R\to Q$ be a surjective morphism of finite quandles. Then
	$$|\pi_0(R)|\le\sum_{[x]\in\pi_0(Q)}|\pi_0(f^{-1}(x))|.$$
\end{lemma}
\begin{proof}
	Let $\m\subseteq Q$ be a set of representatives for $\pi_0(Q)$. For every $x\in Q$, let $\widetilde\psi_x\in\Inn(Q)$ be s.t.
	$$\m_x:=\widetilde\psi_x(x)\in\m.$$
	By \cref{Prp: surj Inn}, we can lift these $\widetilde\psi_x\in\Inn(Q)$ to some $\psi_x\in\Inn(R)$, s.t.
	$$\psi_x\colon f^{-1}(x)\iso f^{-1}(\m_x).$$
	Therefore the map
	$$\coprod_{[x]\in\pi_0(Q)}f^{-1}(x)\to \Inn(Q)$$
	is surjective. Moreover, for every $[x]\in\pi_0(Q)$, the composition
	$$f^{-1}(x)\to R\to Q\to\pi_0(Q)$$
	factors through $f^{-1}(x)\to\pi_0(f^{-1}(x))$. Thus the map
	$$\coprod_{[x]\in\pi_0(Q)}\pi_0(f^{-1}(x))\longrightarrow\pi_0(Q)$$
	is surjective. Hence the desired inequality.
\end{proof}

\begin{proposition}
	\label{Prp: pi0 mul}
	Let $Q,R\in\Qdl^\fin$. Then there is a natural bijection 
	$$\pi_0(Q\times R)\iso\pi_0(Q)\times\pi_0(R).$$
\end{proposition}

\begin{proof}
	If either $Q$ or $R$ is empty, then
	$$\pi_0(Q\times R)=\pi_0(Q)\times\pi_0(R)=\emptyset$$
	and the statement holds. Assume then that $Q,R\neq\emptyset$. Then both projections
	$$Q\times R\to Q\;,\;\;Q\times R\to R$$
	are surjective. By \cref{Prp: surj Inn} we have induced group homomorphisms
	$$\Inn(Q\times R)\to\Inn(Q)\;,\;\;\Inn(Q\times R)\to\Inn(R),$$
	and a well-defined, natural map
	$$\pi_0(Q\times R)=(Q\times R)/\Inn(Q\times R)\to (Q\times R)/\Inn(Q)\times \Inn(R)=\pi_0(Q)\times \pi_0(R).$$
	This map is surjective, as it is induced by taking quotients on $Q\times R\xrightarrow{=} Q\times R$. Next, note that for all $x\in Q$,
	$$\{x\}\times R\to R$$
	is an isomorphism of quandles. We apply \cref{Lma: pi_0 fibration bound} to the projection map $Q\times R\to Q$:
	$$|\pi_0(Q\times R)|\le \sum_{[x]\in\pi_0(Q)}|\pi_0(\{x\}\times R)|=|\pi_0(Q)|\cdot |\pi_0(R)|.$$
	All sets involved are finite because $Q,R\in\Qdl^\fin$, therefore the natural map
	$$\pi_0(Q\times R)\to\pi_0(Q)\times \pi_0(R)$$
	is a bijection.
\end{proof}

\begin{theorem}
	\label{Thm: dim mul}
	Let $Q,R\in\Qdl^\fin$ be finite quandles. Then
	$$\dim_{Q\times R}=\dim_Q\cdot\dim_R.$$
\end{theorem}
\begin{proof}
	We first show that $$\dim_{Q\times R}\ge \dim_Q\cdot \dim_R.$$
	Let $Q'\le Q, R'\le R$. Then $Q'\times R'\le Q\times R$, and by \cref{Prp: pi0 mul},
	$$|\pi_0(Q'\times R')|=|\pi_0(Q')|\cdot |\pi_0(R')|.$$
	Therefore
	$$\dim_{Q\times R}=\max_{D\le Q\times R}|\pi_0(D)|\ge \max_{\substack{Q'\le Q\\R'\le R}}|\pi_0(Q'\times R')|=$$
	$$=\max_{Q'\le Q}|\pi_0(Q')| \cdot \max_{R'\le R}|\pi_0(R')| =\dim_Q\cdot \dim_R.$$
Next we show that
$$\dim_{Q\times R}\le \dim_Q\cdot \dim_R.$$
	Let $D\le Q\times R$. Consider the composition
	$$f\colon D\to Q\times R\xrightarrow{proj_Q} Q.$$
	Denote by $Q'=f(Q)$, a sub-quandle of $Q$ from \cref{Prp: im sub}. By definition,
	$$|\pi_0(Q')|\le\dim_Q.$$
	For all $x\in Q'$, denote by
	$$D_x:=f^{-1}(x)=\{x\}\times R_x$$
	for some subset $R_x\subseteq R$. Since the quandle structure on $D$ is induced by $Q\times R$, and by \cref{Prp: fib sub} we know that $D_x\le D$, it follows that $R_x\le R$, and
	 $$|\pi_0(D_x)|=|\pi_0(R_x)|\le\dim_R.$$
	 We invoke \cref{Lma: pi_0 fibration bound} with the surjective $f\colon D\to Q'$ and find that
	 $$|\pi_0(D)|\le \sum_{\pi_0(Q')}|\pi_0(D_x)|\le \dim_Q\cdot \dim_R.$$
	 This holds for all $D\le Q\times R$, therefore
	 $$\dim_{Q\times R}=\max_{D\le Q\times R}|\pi_0(D)|\le \dim_Q\cdot \dim_R.$$
	 We have the desired equality:
	 $$\dim_{Q\times R}=\dim_Q\cdot \dim_R.$$
\end{proof}

\begin{proposition}
	\label{Prp: prod c prod qdl}
	Let $Q,R\in\Qdl^\fin$. Then for all $n\in\NN$ there is an equality of random variables on $\widehat{B_n}$:	$$c_{Q}\cdot c_{R}= c_{Q\times R}$$
\end{proposition}
\begin{proof}
	By \cref{Prp: qdl prod}, $Q\times R$ is the direct product in $\Qdl$. Therefore for any oriented link $L$,
	$$\Col_{Q\times R}(L)= \Hom_\Qdl(Q_L,Q\times R)\simeq$$
	$$\simeq\Hom_\Qdl(Q_L,Q) \times \Hom_\Qdl(Q_L,R) =\Col_Q(L)\times \Col_R(L).$$
	So for every braid $\sigma\in B_n$, 
	$$\col_{Q\times R}(\bclose\sigma)= \col_{Q}(\bclose\sigma)\cdot \col_{R}(\bclose\sigma).$$
	It follows that the random variables $c_{Q\times R}$ and $c_{Q}\cdot c_{R}$ - both continuous - coincide on $B_n\subseteq\widehat{B_n}$, which is dense. It follows that
	$$ c_{Q\times R} = c_{Q}\cdot  c_{R}$$
\end{proof}

\begin{remark}
	A consequence of \cref{Thm: Qdl HPoly} and \cref{Prp: prod c prod qdl} is that more refined measurements of such $c_{Q,n}$ also coincide with a polynomial for $n\gg 0$. These include higher moments of $c_{Q,n}$, or covariances $\Cov(c_{Q,n},c_{R,n})$. The threshold for $n$ where these measurements coincide with a polynomial is not universal.
\end{remark}

\begin{example}
	A finite nonempty quandle $Q$ is said to be \emph{stable} if the sequence $|\A_{Q,n}|$ is constant for $n\gg 0$, that is if $\deg P_Q=0$. By \cref{Thm: dim deg 1}, $Q\in\Qdl^\fin$ is stable iff $\dim P_Q=1$. By \cref{Thm: dim mul}, the product of stable quandles is again stable, since 
	$$\dim_Q=\dim_R=1\Longrightarrow \dim_{Q\times R}=\dim_Q\cdot\dim_R=1.$$
\end{example}

\section{Disjoint Unions of Quandles}
\begin{proposition}
	Let $Q,R\in\Qdl$ be quandles. The disjoint union $Q\sqcup R$ has a natural quandle structure given by
	$$x,y\in Q\sqcup R\;\;:\;\;\;\qq xy\;\;=\;\;\begin{tabular}{|c||c|c|}
	\hline
	${\qq xy}$&$x\in Q$&$x\in R$\\
	\hline
	\hline
	$y\in Q$&$\qq xy$&$y$\\
	\hline
	$y\in R$&$y$&$\qq xy$\\
	\hline
	\end{tabular}$$
\end{proposition}
\begin{proof}
	Quandle axioms (Q1) and (Q2) clearly hold. As for (Q3):
	$$\qq z{(\qq xy)}=\qq{(\qq zx)}{(\qq zy)}.$$
	If $x,y,z\in Q$ or $x,y,z\in R$, this holds from (Q3) in $Q$ and $R$ separately.\\
	If $x,y\in Q$ and $z\in R$, or likewise if $x,y\in R$ and $z\in Q$, then
	$$\qq z{(\qq xy)}=\qq xy=\qq{(\qq zx)}{(\qq zy)}.$$
	If $x\in Q$ and $y\in R$, or likewise if $x\in R$ and $y\in Q$, then
	$$\qq z{(\qq xy)}=\qq zy=\qq{(\qq zx)}{(\qq zy)}.$$
\end{proof}

The following lemma about monoids is fairly straightforward. The proof is omitted.

\begin{lemma}
	\label{lma: stupid monoid lemma}
	Let $A,B\in\Mon$ be monoids given by generators and relations: $$A=\NN\langle S_A\;|\;R_A\rangle\;,\;\; B=\NN\langle S_B\;|\;R_B\rangle.$$
	Then the free product $A\star B$, which is the coproduct in $\Mon$, is given by
	$$A\star B\simeq \NN\langle S_A\sqcup S_B\;|\;R_A\sqcup R_B\rangle$$
	and the product $A\times B$ is given by 
	$$A\times B\simeq \NN\langle S_A\sqcup S_B\;|\;R_A\sqcup R_B\sqcup\{xy=yx\;|\;\forall x\in S_A,y\in S_B\}\;\rangle.$$
	The natural map $A\star B\to A\times B$ in $\Mon$ from coproduct to product, that comes from the fact that $\Mon$ is a pointed category, is seen here by taking further quotients in RHS.
\end{lemma}

\begin{proposition}
	\label{Prp: A_Q mul on disjoint union}
	Let $Q,R\in\Qdl$. Then there is a natural isomorphism
	$$\A_Q\times \A_R\iso \A_{Q\sqcup R}$$
	in $\Mongr$.
\end{proposition}
\begin{proof}
From the natural embeddings of $Q$ and $R$ into $Q\sqcup R$ we get morphisms
$$\A_Q\to \A_{Q\sqcup R}\;,\;\; \A_R\to \A_{Q\sqcup R}$$
in $\Mongr$.
The coproduct in $\Mongr\simeq\Mon_{/\NN}$ is the coproduct in $\Mon$, which is the free product. Thus we get
$$\A_Q\star\A_R\to\A_{Q\sqcup R}$$
in $\Mongr$.

$$\A_{Q\sqcup R}=\NN\langle Q\sqcup R\;|\;\forall x,y\in Q\sqcup R\;,\;\;x\cdot y=\qq xy\cdot x\rangle=$$
$$= \NN\left\langle Q\sqcup R\;\Bigg|\;
	\begin{matrix}
	\forall x,y\in Q& x\cdot y=\qq xy \cdot x\\
	\forall x,y\in R& x \cdot y=\qq xy \cdot x\\
	\forall x\in Q,y\in R&x \cdot y=y \cdot x\\
	\forall x\in R, y\in Q&x \cdot y=y \cdot x
\end{matrix}\;\right\rangle.$$
By \cref{lma: stupid monoid lemma}, this is naturally isomorphic to
$$\NN\langle Q\;|\;\forall x,y\in Q,\;x\cdot y=\qq xy\cdot x\;\rangle \times \NN\langle R\;|\;\forall x,y\in R,\;x\cdot y=\qq xy\cdot x\;\rangle = \A_Q\times\A_R.$$
This isomorphism preserves generators, all of degree $1$:
$$\A_{Q\sqcup R}\supseteq Q\sqcup R\;\longleftrightarrow\;
\left(Q\times \{1_{\A_R}\}\sqcup \{1_{\A_Q}\}\times R\right) \subseteq \A_Q\times \A_R,$$
thus the isomorphism is one of graded monoids.
\end{proof}

\begin{definition}
	\label{Dfn: genfunc}
	Let $Q\in\Qdl^\fin$. We define the \emph{generating function} for $Q$ as
	$$\genfunc_Q(t)=\sum_{t=0}^\infty |\A_{Q,n}|\cdot t^n\in \ZZ[t].$$
\end{definition}

\begin{proposition}
	\label{Prp: genfunc mul on disjoint union}
	Let $Q,R\in\Qdl^\fin$. Then
	$$\genfunc_{Q\sqcup R}(t)=\genfunc_Q(t)\cdot\genfunc_R(t).$$
\end{proposition}
\begin{proof}
By \cref{Prp: A_Q mul on disjoint union}, $ \A_Q\times\A_R$ and $\A_{Q\sqcup R}$ are isomorphic as graded monoids. For all $n$ we therefore have
	$$A_{Q\sqcup R,n}\simeq
	\left(\A_Q\times\A_R\right)_n\simeq \coprod_{i+j=n}\A_{Q,i}\times\A_{R,j}.$$
	Therefore
	$$\genfunc_Q(t)\cdot \genfunc_R(t)=\left(\sum_i|\A_{Q,i}|t^i\right) \left(\sum_j|\A_{R,j}|t^j\right)=$$
	$$= \sum_n\sum_{i+j=n}|\A_{Q,i}|\cdot |\A_{R,j}|\cdot t^n =\sum_n|\A_{Q\sqcup R,n}|\cdot t^n=\genfunc_{Q\sqcup R}(t).$$
\end{proof}

\begin{proposition}
	\label{Prp: dim via pole ord}
	Let $Q\in\Qdl^\fin$. Then
	$$\eta_Q(t)\in
	\ZZ[t,(1-t)^{-1}]\;,\;\;\ord_{(1-t)}(\eta_Q)=-\dim_Q. $$
\end{proposition}
\begin{proof}
	If $Q=\emptyset$, then $\eta_Q(t)=1$. Otherwise since $P_Q(x)$ is an integer-valued polynomial of degree $\dim_Q-1$, there exist $a_1,\dots a_{\dim_Q}\in\ZZ$ with $a_{\dim_Q}\neq 0$ s.t.
	$$P_Q(x)=\sum_{k=1}^{\dim_Q}a_k\cdot {x+k-1 \choose k-1}.$$
	Thus
	$$\sum_n P_Q(n)t^n=\sum_{k=1}^{\dim_Q}a_k\sum_{n=0}^\infty {n+k-1 \choose k-1}t^n=\sum_{k=1}^{\dim_Q}\frac{a_k}{(1-t)^k}.$$
	Since $|\A_{Q,n}|=P_Q(n)$ for all $n\gg 0$, it follows that
	$$\eta_Q(t)-\sum_{k=1}^{\dim_Q}\frac{a_k}{(1-t)^k}=\eta_Q(t)-\sum_n P_Q(n)t^n=\sum_{n=0}^\infty\left(|\A_{Q,n}|-P_Q(n)\right)t^n\in\ZZ[t].$$
	Thus $\eta_Q(t)\in\ZZ[t,(1-t)^{-1}]$. Since $a_{\dim_Q}\neq 0$, $\ord_{(1-t)}(\eta_Q(t))=-\dim_Q$.
\end{proof}

\begin{remark}
	As an added bonus, one recovers from $\eta_Q(t)$ the threshold  beyond which $|A_{Q,n}|$ agrees with $P_Q(n)$:
	$$\min\{n_0\in\NN\;|\;\forall n > n_0,\;|\A_{Q,n}|=P_Q(n)\}=-\ord_{t=\infty}\eta_Q(t).$$
\end{remark}

\begin{proposition}
	Let $Q,R\in\Qdl^\fin$. Then
	$$\dim_{Q\sqcup R}=\dim_Q+\dim_R.$$
\end{proposition}
\begin{proof}
	By \cref{Prp: dim via pole ord}, we can recover $\dim_Q$ from $\genfunc_Q$ via:
	$$\dim_Q=-\ord_{(t-1)}(\genfunc_Q(t)).$$
	By \cref{Prp: genfunc mul on disjoint union},
	$$\genfunc_{Q\sqcup R}(t)=\genfunc_Q(t)\cdot\genfunc_R(t).$$
	Thus
	$$\dim_{Q\sqcup R}=-\ord_{(t-1)}(\genfunc_{Q\sqcup R}(t))=
	-\ord_{(t-1)}(\genfunc_Q(t)\cdot\genfunc_R(t))=$$
	$$=-\ord_{(t-1)}(\genfunc_Q(t))-\ord_{(t-1)}(\genfunc_R(t))=\dim_Q+\dim_R.$$
\end{proof}

\begin{example}
	Let $Q\in\Qdl^\fin$. Let $Q_+\in\Qdl$ denote $Q_+:=Q\sqcup *$, with $*:=T_1$.
	Then
	$$\genfunc_{Q_+}(t)=\genfunc_{Q}(t)\cdot \frac1{1-t}\;\Longrightarrow\;\genfunc_Q(t)=(1-t)\cdot\genfunc_{Q_+}(t).$$
	Therefore for $n\gg0$,
	$$P_Q(n)=|\A_{Q,n}|=|\A_{Q_+,n}|-|\A_{Q_+,n-1}|=P_{Q_+}(n)-P_{Q_+}(n-1).$$
	Let $\genfunc_Q^\dom(t)\in\ZZ[[t]]$ denote the series
	$$\genfunc_Q^\dom(t):=\sum_{n=0}^\infty \left|\A_{Q,n}^\dom\right|t^n.$$
	Then
	$$\A_{Q_+}^\dom\simeq\A_Q^\dom\times \A_{T_1}^\dom\;\;\Longrightarrow\;\;\genfunc_{Q_+}^\dom(t)=\genfunc_Q^\dom(t)\cdot \frac{t}{1-t}$$
	Therefore for $n\gg0$,
	$$P_Q^\dom(n)=|\A_{Q,n}^\dom|=|\A_{Q_+,n+1}^\dom|-|\A_{Q_+,n}^\dom|=P_{Q_+}(n+1)-P_{Q_+}(n).$$
\end{example}

\subsection{Twisted Pointed Quandles}
In this section, $Q\in\Qdl^\fin$ is a finite quandle and $\psi\in\Aut_\Qdl(Q)$ is such that for all $x\in Q$,
$$\varphi_{\psi(x)}=\varphi_x.$$
\begin{proposition}
	\label{Prp: quotient by psi quandle}
	Let $Q$ and $\psi$ be as above. Then the quotient $Q/\psi$ as a set inherits the quandle structure on $Q$, s.t.
	$Q\to Q/\psi$ is a morphism in $\Qdl$.
\end{proposition}
\begin{proof}
	For all $x,y\in Q$,
	$$\qq{\psi(x)}{y}=\qq xy\;\;,\;\;\;\;\textrm{and}\;\;\;\;\qq{x}{\psi(y)}=\psi\left(\qq{\psi^{-1}(x)}{y}\right)=\psi(\qq xy).$$
	Hence the quandle structure on $Q$ is well defined modulo $\psi$.
\end{proof}

\begin{proposition}
	\label{Prp: twisted by psi pointed quandle}
	Let $Q\in\Qdl$ and $\psi\in\Aut_\Qdl(Q)$ be as above. Then
	$$\begin{tabular}{|c||c|c|}
	\hline
	${\qq xy}$&$x=*$&$x \in Q$\\
	\hline
	\hline
	$y=*$&$*$&$*$\\
	\hline
	$y\in Q$&$\psi(y)$&$\qq xy$\\
	\hline
	\end{tabular}$$
	is a quandle structure on the set $Q\sqcup *$. We denote this by $Q\sqcup_\psi *\in\Qdl$. The canonical embedding of $Q$ into $Q\sqcup_\psi *$ is a morphism in $\Qdl$.
\end{proposition}
\begin{proof}
	Axioms (Q1) and (Q2) are easily verified. As for (Q3): let $x,y,z\in Q\sqcup_\psi *$. Then
	$$z=*\;\;\;\Longrightarrow\;\;\; \qq x{(\qq y*)}=*=\qq{(\qq xy)}{(\qq x*)}.$$
	$$x=*\neq y,z\;\;\;\Longrightarrow\;\;\; \qq *{(\qq yz)}=\psi(\qq yz)=\qq{\psi(y)}{\psi(z)}=\qq{(\qq *y)}{(\qq *z)}.$$
	$$y=*\neq x,z\;\;\;\Longrightarrow\;\;\; \qq x{(\qq *z)}=\qq x{\psi(z)}=\qq{\psi(x)}{\psi(z)}=\psi(\qq xz)=\qq *{(\qq xz)}=\qq{(\qq x*)}{(\qq xz)}.$$
		$$x=y=*\neq z\;\;\;\Longrightarrow\;\;\; \qq *{(\qq *z)}=  \qq{(\qq **)}{(\qq *z)}.$$
\end{proof}

\begin{proposition}
	Let $Q$ and $\psi$ be as above.
	Then the sub-quandles of $Q\sqcup_\psi *$ are the sub-quandles $R\le Q$ and $R\sqcup_\psi *$, where $R\le Q$ satisfies $\psi(R)=R$.
\end{proposition}

\begin{proposition}
	Let $Q\in\Qdl^\fin$ and  let $\psi\in\Aut_\Qdl(Q)$ be as above.
	Then
	$$\genfunc_{Q\sqcup_\psi *}(t)=\genfunc_Q(t)+t\cdot\genfunc_{(Q/\psi)_+}(t),$$
	$$P_{Q\sqcup_\psi *}(x)=P_Q(x)+P_{(Q/\psi)_+}(x-1),$$
	$$P_{Q\sqcup_\psi *}^\dom(x)=P_{(Q/\psi)_+}^\dom(x)$$
\end{proposition}
\begin{proof}
	First extend the $\psi$-action on $Q$ to $Q\sqcup_\psi*$ s.t. $\psi(*)=*$. Then $\psi$ satisfies the above conditions acting on $Q\sqcup_\psi *$ because $\varphi_{\psi(*)}=\varphi_*$ and for all $x\in Q$, $\varphi_{\psi(x)}(*)=*=\varphi_{x}(*)$. By \cref{Prp: quotient by psi quandle} there is a surjective morphism of quandles
	$$f\colon Q\sqcup_\psi *\twoheadrightarrow (Q\sqcup_\psi *)/\psi\simeq (Q/\psi)\sqcup *=(Q/\psi)_+.$$
	Hence there are commutative squares in $\Qdl$ and $\Mongr$ respectively:
	$$\begin{tikzcd}
	Q\sqcup_\psi *\ar[r,two heads]&(Q/\psi)_+\\
	Q\ar[r,two heads]\ar[u,hook']&Q/\psi\ar[u,hook']
	\end{tikzcd}\;\Longrightarrow\;
	\begin{tikzcd}
	\A_{Q\sqcup_\psi*} \ar[r,two heads]&\A_{(Q/\psi)_+}\\
	\A_{Q}\ar[r, two heads]\ar[u,hook']&\A_{Q/\psi}\ar[u,hook']
	\end{tikzcd}.$$
	For all $a\in \A_{Q\sqcup_\psi*}$, $*|a$ iff $*|\A_f(a)$ in $\A_{(Q/\psi)_+}$. Therefore $\A_f$ maps
	$$\begin{tikzcd}
	\A_{Q\sqcup_\psi*}
	\setminus \A_{Q}\ar[r,two heads, "\A_f"']& \A_{(Q/\psi)_+}
	\setminus \A_{Q/\psi}\\
	\{a\in \A_{Q\sqcup_\psi*}\;|\;*|a\}\ar[u,symbol={=}] &
	\{a\in \A_{(Q/\psi)_+}\;|\;*|a\}\ar[u,symbol={=}] 
\end{tikzcd}$$
	surjectively. But since in $\A_{Q\sqcup_\psi *}$ we have for all $x\in Q$,
	$$x\cdot *=\qq x*\cdot x=*\cdot x=\qq *x\cdot *=\psi(x)\cdot *,$$
	this map is a bijection:
	$$\A_f\colon\A_{Q\sqcup_\psi*}\setminus \A_{Q}
	\iso \A_{(Q/\psi)_+}\setminus \A_{Q/\psi}.$$
	Therefore
	$$\genfunc_{Q\sqcup_\psi*}(t)-\genfunc_Q(t)=\genfunc_{(Q/\psi)_+}(t)-\genfunc_{Q/\psi}(t).$$
	But $\genfunc_{Q/\psi}(t)=(1-t)\genfunc_{(Q/\psi)_+}(t)$, therefore
	$$\genfunc_{Q\sqcup_\psi*}(t)=\genfunc_Q(t)+\genfunc_{(Q/\psi)_+}(t)-\genfunc_{Q/\psi}(t)= \genfunc_Q(t)+t\cdot \genfunc_{(Q/\psi)_+}(t).$$
	For $n\gg 0$, 
	$$P_{Q\sqcup_\psi *}(n)=|\A_{Q\sqcup_\psi*,n}|=|\A_{Q,n}|+|\A_{(Q/\psi)_+,n-1}|=P_Q(n)+P_{(Q/\psi)_+}(n-1)$$
	therefore
	$$P_{Q\sqcup_\psi *}(x)=P_Q(x)+P_{(Q/\psi)_+}(x-1).$$
	The sub-quandles of $Q\sqcup_\psi*$ are $\{R\le Q\}\cup\{R\sqcup_\psi*\;|\;R\le Q,\psi(R)=R\}$. Therefore
	$$\A_{Q\sqcup_\psi*}^\dom=\A_{Q\sqcup_\psi*}\setminus \A_Q\setminus\bigcup_{\substack{R< Q\\\psi(R)=R}}\A_{R\sqcup_\psi*}^\dom.$$
	Using $\A_f$ and complete induction on $|Q|$, we have
	$$\A_{Q\sqcup_\psi*}^\dom\simeq \A_{(Q/\psi)_+}\setminus\A_{Q/\psi}\setminus\bigcup_{\substack{R< Q\\\psi(R)=R}}\A_{(R/\psi)_+}^\dom=$$
	$$=\A_{(Q/\psi)_+}\setminus\A_{Q/\psi}\setminus \bigcup_{R<Q/\psi}\A_{R_+}^\dom =\A_{(Q/\psi)_+}^\dom.$$	For $n\gg 0$,
	$$P_{Q\sqcup_\psi*}^\dom(n)=|\A_{Q\sqcup_\psi*,n}|=|\A_{Q/\psi,n}|=P_{(Q/\psi)_+}(n),$$
	therefore
	$$P_{Q\sqcup_\psi*}^\dom(x)=P_{(Q/\psi)_+}^\dom(x).$$
\end{proof}

\section{Examples and Computations}
\subsection{Trivial Quandles}
For $Q=T_0=\emptyset$, $\A_\emptyset=\{1\}$. Since $\emptyset$ has no proper sub-quandles, $\A_\emptyset^\dom=\{1\}$ as well. therefore
$$\genfunc_{T_0}(t)=\genfunc_{T_0}^\dom(t)=1.$$
For $\alpha\ge 1$, $T_\alpha=(T_{\alpha-1})_+$, therefore
$$\genfunc_{T_\alpha}(t)=\genfunc_{T_{\alpha-1}}(t)\cdot \frac1{1-t}\;\;,\;\;\;\;
\genfunc_{T_\alpha}^\dom(t)=\genfunc_{T_{\alpha-1}}^\dom(t)\cdot \frac t{1-t}.$$
By induction on $\alpha$ one obtains
$$\genfunc_{T_\alpha}(t)=\frac1{(1-t)^\alpha}\;\;,\;\;\;\;\genfunc_{T_{\alpha}}^\dom(t)= \frac{t^\alpha}{(1-t)^\alpha},$$
hence
$$P_{T_\alpha}(x)=\textstyle{{x+\alpha-1\choose\alpha-1}}\;\;,\;\;\;\; P_{T_\alpha}(x)=\textstyle{{x-1\choose\alpha-1}}.$$

\subsection{Twisted Pointed Quandles}
Let $Q=T_\alpha$, and let $\psi\colon T_\alpha\to T_\alpha$ be any permutation. Then $\psi\in\Aut_\Qdl(T_\alpha)$ and $\varphi_{\psi(x)}=\varphi_x=\Id$ for all $x\in T_\alpha$. Thus $T_\alpha\sqcup_\psi *$ is a quandle. Let $\beta:=1+|T_\alpha/\psi|$, such that
$(T_\alpha/\psi)_+=T_\beta$. Then $T_\alpha\sqcup_\psi *$ will satisfy
$$\genfunc_{T_\alpha\sqcup_\psi *}(t)=\genfunc_{T_\alpha}(t)+t\cdot \genfunc_{T_\beta}(t)=\frac1{(1-t)^\alpha}+\frac t{(1-t)^\beta}$$
$$P_{T_\alpha\sqcup_\psi *}(x)=P_{T_\alpha}(x)+P_{T_\beta}(x-1)={\textstyle{x+\alpha-1 \choose \alpha-1}}+ {\textstyle{x+\beta-2 \choose \beta-1}} $$
$$P_{T_\alpha\sqcup_\psi *}^\dom(x)=P_{T_\beta}^\dom(x)= {\textstyle{x-1 \choose \beta-1}} $$

\subsubsection{$J$}
Let $Q=T_2=\{a,a'\}$ and let $\psi\in S_2$ be the nontrivial permutation.
Then $J:=T_2\sqcup_\psi \{b\}\in\Qdl$ has structure given by
$$J=\{a,a',b\}\;\;,\;\;\;\;\begin{tabular}{|cc||c|c|c|}
\hline
\multicolumn{2}{|c||}{\multirow{2}{*}{$\qq xy$}}&\multicolumn{3}{|c|}{$x$}\\
\cline{3-5}
&&$a$&$a'$&$b$\\
\hline
\hline
\multirow{3}{*}{$y$}&\multicolumn{1}{|c||}{$a$}&
$a$&$a$&$a'$\\
\cline{2-5}
& \multicolumn{1}{|c||}{$a'$}&
$a'$&$a'$&$a$\\
\cline{2-5}
& \multicolumn{1}{|c||}{$b$}&
$b$&$b$&$b$\\
\hline
\end{tabular}\;\;.$$

This $J$ was introduced in {\cite[\textsection6]{ARTICLE:Joyce1982}} as an example of a quandle that doesn't embed into $\Conj(G)$ for any group $G$.

$$\genfunc_J(t)= \frac1{(1-t)^2}+\frac t{(1-t)^2} =\frac{1+t}{(1-t)^2},$$
$$P_J(t)= {\textstyle{x+1 \choose 1}}+ {\textstyle{x \choose 1}}=2x+1,$$
$$P_J^\dom(x)=P_{(T_2/\psi)_+}^\dom(x)= {\textstyle{x-1 \choose 1}}=x-1.$$

\subsubsection{$J_+$}
Let $Q=T_3=\{a,a',c\}$ and let $\psi\colon Q\to Q$ permute $a\leftrightarrow a'$, fixing $c$. Then $J'=T_3\sqcup_\psi \{b\}$ has structure given by

$$J_+:=T_3\sqcup_{\psi}*\simeq J\sqcup \{c\}=\{a,a',b,c\}\;\;,\;\;\;\;\begin{tabular}{|cc||c|c|c|c|}
\hline
\multicolumn{2}{|c||}{\multirow{2}{*}{$\qq xy$}}&\multicolumn{4}{|c|}{$x$}\\
\cline{3-6}
&&$a$&$a'$&$b$&$c$\\
\hline
\hline
\multirow{4}{*}{$y$}&\multicolumn{1}{|c||}{$a$}&
$a$&$a$&$a'$&$a$\\
\cline{2-6}
& \multicolumn{1}{|c||}{$a'$}&
$a'$&$a'$&$a$&$a'$\\
\cline{2-6}
& \multicolumn{1}{|c||}{$b$}&
$b$&$b$&$b$&$b$\\
\cline{2-6}
& \multicolumn{1}{|c||}{$c$}&
$c$&$c$&$c$&$c$\\
\hline
\end{tabular}\;\;.$$

$$\genfunc_{J_+}(t)= \frac1{(1-t)^3}+\frac t{(1-t)^3} =\frac{1+t}{(1-t)^3},$$
$$P_{J_+}(t)= {\textstyle{x+2 \choose 2}}+ {\textstyle{x+1 \choose 2}} =(x+1)^2,$$
$$P_{J_+}^\dom(x)= {\textstyle{x-1 \choose 2}}=\frac{x^2-3x+2}2.$$

\subsubsection{$C_3$}
Let $Q=T_3=\{a,a',a''\}$ and let $\psi\colon Q\to Q$ be a length-$3$ cycle. Then $C_3=\{a,a',a''\}\sqcup_\psi \{d\}$ has structure given by
$$C_3:=T_3\sqcup_\psi \{d\}=\{a,a',a'',d\}\;\;,\;\;\;\;\begin{tabular}{|cc||c|c|c|c|}
\hline
\multicolumn{2}{|c||}{\multirow{2}{*}{$\qq xy$}}&\multicolumn{4}{|c|}{$x$}\\
\cline{3-6}
&&$a$&$a'$&$a''$&$d$\\
\hline
\hline
\multirow{4}{*}{$y$}&\multicolumn{1}{|c||}{$a$}&
$a$&$a$&$a$&$a'$\\
\cline{2-6}
& \multicolumn{1}{|c||}{$a'$}&
$a'$&$a'$&$a'$&$a''$\\
\cline{2-6}
& \multicolumn{1}{|c||}{$a''$}&
$a''$&$a''$&$a''$&$a$\\
\cline{2-6}
& \multicolumn{1}{|c||}{$d$}&
$d$&$d$&$d$&$d$\\
\hline
\end{tabular}\;\;.$$

$$\genfunc_{C_3}(t)= \frac1{(1-t)^3}+\frac t{(1-t)^2} =\frac{1+t-t^2}{(1-t)^3},$$
$$P_{C_3}(t)= {\textstyle{x+2 \choose 2}}+ {\textstyle{x \choose 1}} =\frac{x^2+5x+2}2,$$
$$P_{C_3}^\dom(x)= {\textstyle{x-1 \choose 1}}=x-1.$$

\subsubsection{$J'$}
Let $Q=J=\{a,a',b\}$ and let $\psi=\varphi_b\in\Aut_\Qdl(J)$. This choice of $Q$ and $\psi\in\Aut_\Qdl(Q)$ satisfy $\varphi_{\psi(x)}=\varphi_x$ for all $x\in Q$, therefore $J\sqcup_\psi \{b'\}\in\Qdl$ with:
$$J':=J\sqcup_\psi\{b'\}=\{a,a',b,b'\}\;\;,\;\;\;\;\begin{tabular}{|cc||c|c|c|c|}
\hline
\multicolumn{2}{|c||}{\multirow{2}{*}{$\qq xy$}}&\multicolumn{4}{|c|}{$x$}\\
\cline{3-6}
&&$a$&$a'$&$b$&$b'$\\
\hline
\hline
\multirow{4}{*}{$y$}&\multicolumn{1}{|c||}{$a$}&
$a$&$a$&$a'$&$a'$\\
\cline{2-6}
& \multicolumn{1}{|c||}{$a'$}&
$a'$&$a'$&$a$&$a$\\
\cline{2-6}
& \multicolumn{1}{|c||}{$b$}&
$b$&$b$&$b$&$b$\\
\cline{2-6}
& \multicolumn{1}{|c||}{$b$}&
$b'$&$b'$&$b'$&$b'$\\
\hline
\end{tabular}\;\;.$$
Here $(J/\psi)_+\simeq T_3$, therefore
$$\genfunc_{J'}(t)=\genfunc_J(t)+t\cdot\genfunc_{T_3}(t)=\frac{1+t}{(1-t)^2}+\frac{t}{(1-t)^3}=\frac{1+t-t^2}{(1-t)^3},$$
	$$P_{J'}(x)=P_J(x)+P_{T_3}(x-1)=(2x+1)+{\textstyle{x+1\choose 2}}=\frac{x^2+5x+2}2,$$
	$$P_{J'}^\dom(x)=P_{T_3}^\dom(x)= {\textstyle{x-1\choose 2}} =\frac{x^2-3x+2}2.$$


\subsection{Dihedral Quandles}
There is a functor from abelian groups to quandles -
$$\Dhd{}\colon\Ab\to\Qdl\;\;,\;\;\;\;M\mapsto \Dhd M,$$
where $D_M$ has underlying set $M$, and quandle structure
$$\qq xy=2x-y.$$
A morphism $f\colon M\to M'$ in $\Ab$ yields a morphism $\Dhd f\colon \Dhd M\to\Dhd{M'}$ in $\Qdl$:
$$\forall x\in M\;\;:\;\;\;\;f(\qq xy)=f(2x-y)=2f(x)-f(y)=\qq{f(x)}{f(y)}.$$
For clarity in computations we distinguish between elements of $M\in\Ab$ and $\Dhd M\in\Qdl$ thusly:
$$x\in M\;,\;\;[x]\in \Dhd M.$$
For every $M\in\Ab$ there is a natural function
$$\T_M\colon \A_{\Dhd M}\to M\;\;,\;\;\;\;\T_M([x_0]\cdots[x_n])=\sum_{i=0}^n (-1)^i\cdot x_i.$$
This is the unique function $\A_{\Dhd M}\to M$ satisfying
$$\forall x\in M\;,\;\; \T_M([x])=x\;,$$
$$\forall \;a,b\in\A_{\Dhd M}\;,\;\;\T_M(a\cdot b)=\T_M(a)+ (-1)^{|a|}\T_M(b).$$
For $M=\ZZ/\ell\ZZ\in\Ab$ these $\Dhd M$ are called \emph{dihedral quandles}, denoted by
$$\Dhd \ell:=\Dhd{\ZZ/\ell\ZZ}.$$
For all $M\in\Ab$, the quandle $\Dhd M$ is naturally isomorphic to the sub-quandle of the dihedral group $M\rtimes \{\pm1\}$ consisting of all reflections, via
$$\Dhd M\to M\rtimes\{\pm1\}\;\;,\;\;\;\;m\mapsto (m,1).$$
The aim here is to compute $P_{D_\ell}(x), P_{D_\ell}^\dom(x)\in\QQ[x]$ and $\genfunc_{D_\ell}(t)\in\ZZ(t)$ for all $\ell\in\NN$. We need some preliminary statements:

\begin{proposition}
	\label{Prp: classify subquandles of dihedral qdl}
	Let $\ell\in\ZZ$ and let $M=\ZZ/\ell\ZZ\in\Ab$. Then the sub-quandles $R$ of $\Dhd \ell$ are precisely $R=\emptyset$ and the cosets of subgroups of $\ZZ/\ell\ZZ$:
	$$\{R\le \Dhd \ell\}=\{\emptyset\}\cup\{c+H\;\;|\;\;\;\;H\le \ZZ/\ell\ZZ\;,\;\;c\in \left(\ZZ/\ell\ZZ\right)/H\}.$$
	For each $c+H\le \Dhd M$, the map
	$$H\to c+ H\;\;,\;\;\;\;h\mapsto c+h$$
	is an isomorphism of quandles.
\end{proposition}
\begin{proof}
	Let $\emptyset\neq R\le D_\ell$. For every $[x]\in R$, let 
	$$m_x:=m_{[x]}:=\min\{n\in\NN\;|\;[x+n]\in R\}.$$
	Consider $[x]\in R$ with minimal $m_{x}$. Then on one hand $$[x+m_{x}]\in R \;\;\;\textrm{ and }\;\;\;m_{x+m_x}\ge m_x.$$
	On the other hand
	$$\qq{[x+m_x]}{[x]}=[2(x+m_x)-x]=[(x+m_x)+m_x]\in R.$$
	Therefore
	$$m_{x+m_x}=m_x.$$
	It follows that $[x+m_x\NN]\subseteq R$. Having chosen $x$ with minimal $m_x$, we find that
	$$R=[x+m_x\NN]=x+m_x\ZZ/\ell\ZZ$$
	is a coset of the subgroup $m_x\ZZ/\ell\ZZ \le \ZZ/\ell\ZZ$. All cosets of a subgroup $H\le \ZZ/\ell\ZZ$ are isomorphic sub-quandles:
	Let $c\in\ZZ/\ell\ZZ$. The map $x\mapsto x+c$ is in $\Aut_\Qdl(\Dhd \ell)$:
	$$\forall x,y\in\ZZ/\ell\ZZ\;,\;\;\qq{(x+c)}{(y+c)}=2(x+c)-(y+c)=(2x-y)+c=\qq xy+c.$$
	Moreover the map is a bijection, therefore an automorphism of $\Dhd \ell$ in $\Qdl$. $\ZZ/\ell\ZZ$ acts transitively on cosets of $H$, mapping one isomorphically onto another.
\end{proof}
\begin{remark}
	The statement of \cref{Prp: classify subquandles of dihedral qdl} is false for arbitrary $M\in\Ab^\fin$ - for all $k$, $\Dhd{(\ZZ/2\ZZ)^k}$ is trivial, so \emph{any} subset is a sub-quandle. If $M$ is odd however, this does hold and the proof is much simpler.
\end{remark}

\begin{proposition}
\label{Prp: Bn dihedral split}
	Let $\ell,\ell'\in\NN$ be coprime. Then for all $n\in\NN$,
	$$|\A_{\Dhd{\ell\ell'},n}|= |\A_{\Dhd{\ell},n}|\cdot |\A_{\Dhd{\ell'},n}|.$$
\end{proposition}
\begin{proof}
	The isomorphism $\ZZ/\ell\ell'\ZZ\iso \ZZ\ell\ZZ\times\ZZ/\ell'\ZZ$ in $\Ab$ carries over to a morphism $f\colon \Dhd{\ell\ell'}\to\Dhd\ell\times\Dhd{\ell'}$ in $\Qdl$. Since $f$ is a bijection, $f$ is an isomorphism.
	For any $Q\in\Qdl^\fin$, let $h_Q\in\NN$ denote
	$$h_Q:=[B_2:\ker(B_2\xrightarrow{\beta}\Aut_\Set(Q^2))].$$
	In the case of $Q=\Dhd\ell$, $\sigma_1(x+c,x)= (x+2c,x+c)$,
	therefore $h_{\Dhd\ell}=\ell$. We show that if $Q,R\in \Qdl^\fin$ are s.t. $h_Q$ and $h_R$ are coprime, then for all $n\in\NN$,
	$$\A_{Q\times R,n}\simeq \A_{Q,n}\times \A_{R,n}.$$
	Since $\ell=h_{\Dhd\ell}$ and $\ell'=h_{\Dhd{\ell'}}$ are coprime, this will finish the proof.
	For $n=0$ both sides are singletons. For $n=1$ both sides are $Q\times R$. Let $n\ge 2$. Because $h_Q,h_R$ are coprime, there are $m,k\in\NN$ s.t.
	$$\begin{matrix}
		m\equiv 0 \mod h_Q\;,&k\equiv 1 \mod h_Q\\
		m\equiv 1 \mod h_R\;,&k\equiv 0 \mod h_R
	\end{matrix}$$
	Then for all $\sigma_i\in B_n$ and all $z_i=(x_i,y_i),\;z_{i+1}=(x_{i+1},y_{i+1})\in Q\times R$,
	$$\begin{array}{c|c}
		\sigma_i^k(\dots x_i,x_{i+1}\dots)=\sigma _i (\dots x_i,x_{i+1}\dots) &
		\sigma _i ^m(\dots x_i,x_{i+1}\dots)= (\dots x_i,x_{i+1}\dots)\\
		\sigma _i ^k(\dots y_i,y_{i+1}\dots)=(\dots y_i,y_{i+1}\dots) &
		\sigma _i ^m(\dots y_i,y_{i+1}\dots)=\sigma _i(\dots y_i,y_{i+1}\dots)
	\end{array}$$
It follows that in the $B_n$-action on $(Q\times R)^n$, we can force an action on $Q^n$ and $R^n$ separately. Therefore for all $n\ge 2$ (and $n=0,1$ from before)
$$\A_{Q\times R,n}\simeq (Q\times R)^n/B_n\simeq Q^n/B_n\times R^n/B_n\simeq\A_{Q,n}\times\A_{R,n}.$$
\end{proof}

In light of, all computations of $|\A_{\Dhd\ell,n}|$ are reduced to their $p$-primary components. We therefore assume in the following computations that
$$\ell=p^r\;\;,\;\;\;\;\textrm{with}\;\;\;p\;\;\textrm{prime}\;,\;\;r>0\;.$$
Furthermore, when headed into computations we distinguish between $x\in\ZZ/\ell\ZZ$ and $[x]\in D_{\ell}$ - inside the brackets we perform standard arithemtic in $\ZZ/\ell\ZZ$.

\begin{lemma}
	\label{Lma: dom iff two diff unit}
	Let $Q=D_\ell$ and let $x,y\in\ZZ/\ell\ZZ$. Then
	$$[y]\cdot[x]\in\A_{\Dhd\ell}^\dom\;\iff\;y-x\in(\ZZ/\ell\ZZ)^\times.$$
	More generally, $[x_1]\cdots [x_n]\in\A_{\Dhd\ell}^\dom$ iff there exist some $1\le i,j\le n$ s.t.
	$$x_j-x_i\in \left(\ZZ/\ell\ZZ\right)^\times.$$
\end{lemma}

\begin{lemma}
	Let $\ell=p^r$, $r\ge 1$ and $Q=\Dhd\ell\in\Qdl$. Let $x,y\in\ZZ/\ell\ZZ$.
	\begin{enumerate}
		\item For all $k\in\ZZ$,
		$$[x+y]\cdot [x] = [x+(k+1)y]\cdot [x+ky].$$
		\label{Lma: translator}
		\item If $y-x\in(\ZZ/\ell\ZZ)^\times$ then for all $k\in\ZZ$
		$$[x]\cdot[y] = [x+k]\cdot [y+k].$$
		\label{Lma: universal translator}
		\item If $x\in(\ZZ/\ell\ZZ)^\times$, then for all $k\in\ZZ$,
		$$[x+pk]\cdot [x]\cdot[0]=[1+pk]\cdot[1]\cdot[0].$$
		\label{Lma: down to one}
	\end{enumerate}
\end{lemma}
\begin{proof}
	\begin{enumerate}
		\item 
		\item Since $y-x\in(\ZZ/\ell\ZZ)^\times$, there exists $m\in\ZZ$ s.t. $m\cdot (y-x)=k\mod \ell$. Therefore
	$$[y]\cdot [x]=[x+(y-x)]\cdot [x]=[x+(m+1)(y-x)]\cdot[x+m(y-x)]= $$
	$$=[y+m(y-x)]\cdot[x+m(y-x)]=[y+k]\cdot [x+k].$$
		\item By \cref{Lma: universal translator},
	$$[x+pc]\cdot[x] \cdot[0]=[x+pc] \cdot[x+1+pc] \cdot[1+pc]=$$
	$$=[1+pc] \cdot[2+pc] \cdot[1+pc]=[1+pc] \cdot[1] \cdot[0].$$
	\end{enumerate}
\end{proof}

\begin{lemma}
	\label{Lma: Dihedral dom div [0]}
	Let $Q=\Dhd\ell$, let $n\ge 2$ and let $a\in\A_{\Dhd\ell,n}^\dom$. Then $[0]$ divides $a$.
\end{lemma}
\begin{proof}
	Write $a=[x_1]\cdots [x_n]\in\A_{\Dhd\ell,n}^\dom$. By \cref{Lma: dom iff two diff unit}, there exist $1\le i<j\le n$ s.t. $x_i-x_j\in\Zlx$.
	By \cref{Lma: MonConj multidiv} and \cref{Lma: universal translator}, there is $a'\in\A_{\Dhd\ell}$ s.t.
	$$a=a'\cdot[x_i]\cdot[x_j]=a'\cdot[x_i-x_j]\cdot[0].$$
	Therefore $[0]$ divides $a$.
\end{proof}

\begin{corollary}
	\label{Cor: Dihedral dom n=2 case}
	Let $\ell=p^r$, $r\ge 0$, and let $Q=\Dhd\ell$. Then
	$$|\A_{\Dhd{\ell},2}^\dom|=|(\ZZ/\ell\ZZ)^\times|=p^r-p^{r-1}.$$
\end{corollary}
\begin{proof}
	By \cref{Lma: dom iff two diff unit}, for all $c\in\ZZ/\ell\ZZ$,
	$$[c]\cdot[0]\in\A_{\Dhd\ell}^\dom.\;\iff\;c\in(\ZZ/\ell\ZZ)^\times.$$
	By \cref{Lma: Dihedral dom div [0]}, every $a\in\A_{\Dhd\ell,2}^\dom$ is of the form $a=[c]\cdot[0]$, for some $c\in\ZZ/\ell\ZZ$ uniquely determined by $a$ via $c=c-0=\T_{\ZZ/\ell\ZZ}(a)$. Thus
	$$|\A_{\Dhd\ell,2}^\dom|=|\{[c]\cdot[0]\;|\;c\in(\ZZ/\ell\ZZ)^\times\}|=| (\ZZ/\ell\ZZ)^\times |=p^r-p^{r-1}.$$
\end{proof}

\begin{proposition}
	\label{Prp: Dihedral odd dom comp}
	Let $\ell=p^r$ be odd, $r\ge 1$, and let $Q=\Dhd\ell$. Then for all $n\ge 3$
	$$|\A_{\Dhd,n}^\dom|=\ell.$$
\end{proposition}
\begin{proof}
	First we show that $[0]^{n-2}|a$ for every $a\in\A_{\Dhd\ell}^\dom$ of degree $n=|a|\ge 3$. By \cref{Lma: Dihedral dom div [0]}, $[0]\;|\;a$. To assert this claim it suffices to show that 
	$$b\cdot [0]\in\A_{\Dhd\ell}^\dom\;,\;\;|b|\ge 3\;\;\Longrightarrow\;\;[0]^2\;|\;b\cdot[0]$$
	- in fact it is enough to show this for $b=[x]\cdot[y]\cdot [z]$ of $|b|=3$. If $b\in\A_{\Dhd\ell}^\dom$, then $[0]\;|\;b$ by \cref{Lma: Dihedral dom div [0]} and we are finished.
	Otherwise we must have $p|(y-z)$. Also note that $x,y,z\in(\ZZ/\ell\ZZ)^\times$ because $b\cdot[0]\in\A_{\Dhd\ell}^\dom$. By \cref{Lma: down to one} it follows that
	$$[y]\cdot [z]\cdot [0]=[y-z+1]\cdot[1]\cdot[0]=[-x+y-z]\cdot[-x]\cdot [0].$$
	Therefore
	$$b\cdot [0]= [x]\cdot [y]\cdot [z]\cdot [0] = [x]\cdot [-x+y-z]\cdot[-x]\cdot [0].$$
	Let $b':=[x]\cdot [-x+y-z]\cdot[-x]$. Since $\ell$ is odd, $x-(-x)=2x\in(\ZZ/\ell\ZZ)^\times$. By \cref{Lma: dom iff two diff unit} $b'\in\A_{\Dhd\ell}^\dom$. By \cref{Lma: Dihedral dom div [0]}, $[0]\;|\;b'$. Therefore
	$$[0]^2\;|\;b\cdot [0].$$
	Now let $n\ge 3$. Then for every $a\in\A_{\Dhd\ell,n}^\dom$ there exist $x,y\in\ZZ/\ell\ZZ$ s.t.
	$$a=[x]\cdot[y]\cdot[0]^{n-2}.$$
	If $[x]\cdot[y]\in\A_{\Dhd\ell}^\dom$ then 
	$$a=[z]\cdot [0]^{n-1}.$$
	Since $a\in\A_{\Dhd\ell}^\dom$ and $n-2\ge 1$, this $z$ must reside in $(\ZZ/\ell\ZZ)^\times$ by \cref{Lma: dom iff two diff unit}. If $[x]\cdot[y]\notin\A_{\Dhd\ell}^\dom$, then $p|(y-x)$ - yet $a\in\A_{\Dhd\ell}^\dom$, therefore $x,y\in(\ZZ/\ell\ZZ)^\times$. By \cref{Lma: down to one},
	$$a=[x]\cdot[y]\cdot[0]^{n-2}=[1+z]\cdot[1]\cdot[0]^{n-2}$$
	for some $z\in p\ZZ/\ell\ZZ$. In both cases $z=x-y=\T_{\ZZ/\ell\ZZ}(a)$ is uniquely determined by $a$. Thus
	$$|\A_{\Dhd\ell,n}^\dom|=|\{[z]\cdot [0]^{n-1}\;|\;z\in(\ZZ/\ell\ZZ)^\times\}\cup\{[1+z]\cdot[1]\cdot[0]^{n-2}\;|\;z\in p\ZZ/\ell\ZZ\}|=$$
	$$=|\ZZ/\ell\ZZ|=\ell.$$
\end{proof}

\begin{proposition}
	\label{Prp: Dihedral even dom comp}
	Let $\ell=2^r$ with $r\ge 1$, and let $n\ge 3$. Then
	$$\left|\A_{\Dhd\ell,n}^\dom\right|=2^{r-1}(n-1).$$
\end{proposition}
\begin{proof}
	First we show that for all $a\in\A_{\Dhd\ell,n}^\dom$, there exist unique $0<j<n$ and $u\in\Zlx$ s.t.
	$$a=[u]\cdot[1]^{n-j-1}\cdot[0]^{j}.$$
	Let $j\in\NN$ be the largest s.t. $[0]\;|\;a$. Since $[0]^n\notin\A_{\Dhd\ell}^\dom$, By \cref{Lma: Dihedral dom div [0]}, $0<j<n$.
	Write
	$$a=a_1\cdot [0]^j.$$
	From the maximality of $j$ and \cref{Prp: classify subquandles of dihedral qdl}, $a_1\in\A_{1+2\ZZ/\ell\ZZ}$. Also $j\ge 1$ because $a$ is dominant.
	For all $i$, write $[1]^i\cdot [0]=[0]\cdot u_i$. By \cref{Lma: down to one}, for all $c,d\in\ZZ/\ell\ZZ$,
	$$[1+2c]\cdot[1+2d]\cdot [1]^i\cdot[0]= [1+2c]\cdot[1+2d]\cdot[0]\cdot u_i =$$
	$$=[1+2c-2d]\cdot[1]\cdot[0]\cdot u_i= [1+2c-2d]\cdot[1]^{i+1}\cdot[0].$$
	Therefore we can write
	$$a = a_1\cdot [0]^{j}=[u]\cdot [1]^{n-j-1}\cdot [0]^j.$$
	Note that $u\in\Zlx$, as $[u]$ divides $a_1\in\A_{1+2\ZZ/\ell\ZZ}$. The presentation of $a\in\A_{\Dhd\ell}^\dom$ as $a=[u]\cdot[1]^{n-j-1}\cdot [0]^j$ is unique: suppose
	$$[u]\cdot[1]^{n-j-1}\cdot [0]^j = [v]\cdot[1]^{n-i-1}\cdot [0]^i\;\;,\;\;\;\;u,v\in\Zlx.$$
	The equality persists in $\A_{\Dhd2}$ in the form $[1]^{n-j}\cdot [0]^{j}=[1]^{n-i}\cdot[0]^i$. Because $\Dhd2$ is a trivial quandle, it follows that $j=i$. The value of $u\in\Zlx$ is then determined by $\T_{\ZZ/\ell\ZZ}(a)$ and $j$. The surjective morphism $\A_{\Dhd\ell}\to\A_{\Dhd2}$ also shows that all values of $0<j<n$ and $u\in\Zlx$ are manifest. Thus
	$$|\A_{\Dhd\ell,n}^\dom|=|\Zlx|\cdot |\{0<j<n\}|=2^{r-1}(n-1).$$
\end{proof}

\begin{proposition}
	\label{Prp: Dihedral total - n ge 3}
	Let $\ell=p^r$, $r\ge 1$ and let $Q=\Dhd\ell$. Then for all $n\ge 3$,
	$$\left|\A_{\Dhd\ell,n}\right|=\begin{cases}
		(r+1)p^r&p\neq 2\\
		2^r+2^{r-1}r(n-1)&p=2
	\end{cases}.$$
\end{proposition}
\begin{proof}
	By \cref{Lma: dom decomp A_Q} and \cref{Prp: classify subquandles of dihedral qdl},
	$$\left|\A_{\Dhd\ell,n}\right|=\sum_{R\le \Dhd\ell}\left|\A_{R,n}^\dom\right|= \sum_{k=0}^rp^{r-k}\left|\A_{\Dhd{p^k},n}^\dom\right|= p^r+ \sum_{k=1}^rp^{r-k}\left|\A_{\Dhd{p^k},n}^\dom\right|.$$
	If $p\neq 2$, by \cref{Prp: Dihedral odd dom comp},
	$$\left|\A_{\Dhd\ell,n}\right| =p^r+\sum_{k=1}^rp^{k-r}\cdot p^k=p^r+rp^r=(r+1)p^r.$$
	If $p=2$, then by \cref{Prp: Dihedral even dom comp},
	$$\left|\A_{\Dhd\ell,n}\right| =2^r+\sum_{k=1}^r2^{k-r}\cdot (2^{k-1}(n-1))=2^r+r2^{r-1}(n-1).$$
\end{proof}

\begin{proposition}
	Let $\ell=p^r$, $r\ge 1$ and let $Q=\Dhd\ell$. Then
	
	$$\begin{array}{|c|c|c|l|}
	\hline
    &   P_{\Dhd\ell}(x)& P_{\Dhd\ell}^\dom(x)&\genfunc_{\Dhd\ell}(t)\\
    \hline
    \hline
    p\neq 2 & (r+1)p^r &p^r& (1-t)^{-1}\left({1+(p^r-1)t+r(p^r-p^{r-1})t^2+rp^{r-1}t^3}\right)\\
    \hline
    p=2& (2^r-r2^{r-1})+r2^{r-1}x&2^{r-1}(x-1)&(1-t)^{-2}\left(1+(2^r-2)t+(r2^{r-1}-2^r+1)t^2\right)\\
    \hline
\end{array}$$
\end{proposition}
\begin{proof}
	The computations of $P_{\Dhd\ell}(x), P_{\Dhd\ell}^\dom(x)$ for $p$ either even or odd, are direct results of propositions \ref{Prp: Dihedral odd dom comp}, \ref{Prp: Dihedral even dom comp} and \ref{Prp: Dihedral total - n ge 3}. For $\genfunc_{\Dhd\ell}(t)$, all that remains to compute is $|\A_{\Dhd\ell,2}|$. By \cref{Cor: Dihedral dom n=2 case}, $|\A_{\Dhd\ell,2}^\dom|=p^r-p^{r-1}$. Therefore
	$$|\A_{\Dhd\ell,2}|=\sum_{k=0}^r p^{r-k}|\A_{\Dhd{p^k},2}^\dom|=p^r+ \sum_{k=1}^r p^{r-k}(p^k-p^{k-1})= $$
	$$= p^r+ r(p^r-p^{r-1})= (r+1)p^r-rp^{r-1}.$$
	For $p\neq 2$ we get
	$$\genfunc_{\Dhd\ell}(t)=1+p^rt-rp^{r-1}t^2+\frac{(r+1)p^r\cdot t^2}{1-t}=$$
	$$=\frac{1+(p^r-1)\cdot t+r(p^r-p^{r-1})\cdot t^2-rp^{r-1}\cdot t^3}{1-t}.$$
	For $p=2$, note that
	$$n=1\;:\;\;|\A_{\Dhd\ell,1}|=\ell=P_{\Dhd\ell}(1)$$
	$$n=2\;:\;\;|\A_{\Dhd\ell,2}|=(r+1)2^r-r2^{r-1}=2^r+r2^{r-1}=P_{\Dhd\ell}(2),$$
	Therefore
	$$\genfunc_{\Dhd\ell}(t)=1+\sum_{n=1}^\infty\left((2^r-r2^{r-1})+r2^{r-1}n\right)t^n=1+\frac{(2^r-r2^{r-1})\cdot t}{1-t}+\frac{r2^{r-1}\cdot t}{(1-t)^2}=$$
	$$=\frac{1+(2^r-2)\cdot t+(r2^{r-1}-2^r+1)\cdot t^2}{(1-t)^2}.$$
\end{proof}

In particular, we obtain $P_Q(x),P_Q^\dom(x)\in\QQ[x]$ and $\genfunc_Q(t)\in\ZZ[t,(1-t)^{-1}]$ pertaining to the following small examples:

\subsubsection{$Q=\Dhd3$}
For $\ell=3$, $\Dhd3$ satisfies
$$P_{\Dhd3}(x)=6\;,\;\;P_{\Dhd3}^\dom(x)=3\;,\;\;\eta_{\Dhd3}(t)=\frac{1+2t+2t^2+t^3}{1-t}.$$

\subsubsection{$Q=\Dhd4$}
For $\ell=4=2^2$, $\Dhd4$ satisfies
$$P_{\Dhd4}(x)=4x\;,\;\;P_{\Dhd4}^\dom(x)=2x-2\;,\;\;\eta_{\Dhd4}(x)=\frac{1+2t+t^2}{(1-t)^2}.$$

\subsubsection{$Q=\Dhd{3,+}$}
$$\eta_{\Dhd{3,+}}(t)=\eta_{\Dhd3}(t)\cdot\eta_{T_1}(t)=\frac{1+2t+2t^2+t^3}{(1-t)^2}= t+4+\frac{6}{(1-t)^2}-9\frac{1}{1-t}=$$
	$$=t+4+\sum_n\left(6{\textstyle{n+1\choose1}}-9 {\textstyle{n+0\choose0}}\right)t^n=t+4+\sum_n(6n-3)t^n$$
Therefore
$$P_{\Dhd{3,+}}(x) =6x-3.$$
The proper sub-quandles $R$ of $D_3$ are $\emptyset$ and three singletons. Therefore
$$P_{Q_+}^\dom(x) = P_{Q_+}(x) - P_Q(x)-\sum_{R<Q}P_{R_+}^\dom(x)=$$
$$=P_{\Dhd{3,+}}(x)-P_{\Dhd3}(x)-3P_{T_2}^\dom(x)-P_{T_1}^\dom(x)=
(6x-3)-6-3(x-1)-1=3x-7.$$

\subsection{Small Quandles}
We list all $Q\in\Qdl$, up to isomorphism, satsifying $|Q|\le 4$ - along with $P_Q(x),P_Q^\dom(x)\in\QQ[x]$ and $\genfunc_Q(t)\in\ZZ[t,(1-t)^{-1}]$.

$$\begin{array}{|c|c|c|c|l|}
	\hline
    |Q|& Q&  P_Q(x)& P_Q^\dom(x)&\genfunc_Q(t)\\
    \hline
    \hline
    0&\emptyset & 0 &0& 1\\
    \hline
    1&T_1&1&1&(1-t)^{-1}\\
    \hline
    2&T_2&x+1&x-1& (1-t)^{-2}\\
    \hline
    &T_3&\frac12(x^2+3x+2)& \frac12(x^2-3x+2)& (1-t)^{-3}\\
    3&J&2x+1&x-1&(1-t)^{-2} (1+t)\\
    &\Dhd3&6&3& (1-t)^{-1}(1+2t+2t^2+t^3)\\
    \hline
    &T_4& \frac16(x^3+6x^2+11x+6)& \frac16(x^3-6x^2+11x-6)&(1-t)^{-4}\\
    &J_+&x^2+2x+1& \frac12(x^2-3x+2)&(1-t)^{-3} (1+t)\\
    4&\Dhd{3,+}&6x-3&3x-7&(1-t)^{-2}(1+2t+2t^2+t^3)\\
    &J’&\frac12(x^2+5x+2)& \frac12(x^2-3x+2)&(1-t)^{-3} (1+t-t^2)\\
    &D_4&4x&2x-2&(1-t)^{-2} (1+2t+t^2)\\
    &C_3&\frac12(x^2+5x+2)&x-1&(1-t)^3(1+t-t^2)\\
    \hline
\end{array}$$


\bibliographystyle{alpha}
\bibliography{Ariel_Davis_bibliography.bib}

\end{document}